\documentclass[leqno,10pt]{article}

\usepackage{fullpage}
\usepackage{latexsym}
\usepackage[centertags]{amsmath}
\usepackage{amssymb,amsthm,amsfonts}
\usepackage{color}
\newtheorem{theorem}{Theorem}
\newtheorem{lemma}[theorem]{Lemma}
\newtheorem{prop}[theorem]{Proposition}
\newtheorem{corl}[theorem]{Corollary}
\newtheorem{defn}[theorem]{Definition}
\theoremstyle{remark}
\newtheorem*{rmk}{Remark}

\numberwithin{theorem}{section}

\newcommand{\bydef}{:=}

\newcommand{\charf}{\mbox{{\text 1}\kern-.24em {\text l}}}

\newcommand{\defby}{=:}
\newcommand{\del}{\partial}
\newcommand{\dist}{{\mbox{dist}}}
\renewcommand{\div}{\operatorname{div}}

\newcommand{\eps}{\varepsilon}

\renewcommand{\geq}{\geqslant}

\newcommand{\lap}{\operatorname{\Delta}}
\newcommand{\lapi}{\operatorname{{\Delta_\infty}}}
\renewcommand{\leq}{\leqslant}

\newcommand{\oo}[1]{\frac{1}{#1}}
\newcommand{\openR}{{{\text I}\kern-.22em {\text R}}}

\newcommand{\R}{{\mathbb R}}
\newcommand{\Rd}{{\R^d}}

\newcommand{\RR}{{\mathbb R}}

\newcommand{\Fto}{{\sideset{_{2}}{_{1}}{\operatorname{F}}}}

\numberwithin{equation}{section}
\begin{document}

\title
{\bf A porous medium equation
involving the infinity-Laplacian. Viscosity solutions and asymptotic behaviour
}

\author{Manuel Portilheiro\footnote{Dpto.\ de Matem\'aticas, Univ.\ Aut\'onoma de Madrid, Spain. \texttt{manuel.portilheiro@uam.es}}\\[4pt]
Juan Luis V\'azquez\footnote{Dpto.\ de Matem\'aticas, Univ.\ Aut\'onoma de Madrid, Spain. \texttt{juanluis.vazquez@uam.es}}}

\date{\today}
\maketitle

\begin{abstract}
We study a nonlinear porous medium type equation involving the infinity Laplacian operator.
We first consider the problem posed on a bounded domain and prove existence of maximal nonnegative viscosity solutions. Uniqueness is obtained for strictly positive solutions with Lipschitz in time data. We also describe the asymptotic behaviour for the Dirichlet problem in the class of maximal solutions.  We then discuss the Cauchy problem posed in the whole space. As in the standard porous medium equation (PME), solutions which start with
compact support exhibit a free boundary propagating with finite speed, but such propagation takes place only in the direction of the spatial gradient. The description of the asymptotic behaviour of the Cauchy Problem shows that the asymptotic profile and the rates of convergence and propagation agree for large times with a one-dimensional PME.
\end{abstract}

\allowdisplaybreaks

\section{Introduction}\label{S:intro}

In order to understand modes of nonlinear diffusion with preferential propagation in some directions, we  consider here the following variation of the porous medium equation,
\begin{equation}\label{E:ipme}
	\rho_t = \lapi (\rho^m),
\end{equation}
where $m>1$ and $\lapi$ denotes the 1-homogeneous version of the infinity-Laplacian,
\[
	\lapi u \bydef |Du|^{-2} \sum_{i,j=1}^d u_{x_ix_j} u_{x_i} u_{x_j}.
\]
This equation is related to the standard porous medium equation (PME): $\rho_t - \Delta(\rho^m) = 0$. It represents a
diffusion of porous medium type, but taking place only in the direction of the spatial gradient of $\rho$. Moreover, if  a solution of \eqref{E:ipme} is radial (i.\,e., $u=u(|x|,t)$), it corresponds exactly to a solution of a $1$-$d$ PME with spatial variable $r=|x|$.

Here we examine the general non-radial case with the intention of proving that there exists a natural evolution process that still looks like a $1$-$d$ directional PME, at least for large times. However, in this case the diffusion chooses the direction of the gradient, that varies with space and time. The first difficulty in devising a theory is to define a suitable concept of generalized solution for this nonlinear equation, which is formally parabolic but can be degenerate or singular at some points: Here we use a  well-known trick of the PME theory, the so-called density-to-pressure transformation, and introduce the variable
\begin{equation}\label{equiv}
	u \bydef \frac{m}{m-1}\, \rho^{m-1}.
\end{equation}
Proceeding formally, it is easy to see that, if everything is smooth, $u$ must solve
\begin{equation}\label{E:mu}
	u_t = k\,u \lapi u + |Du|^2
\end{equation}
with $k = m-1>0$. We propose to  work out the existence and uniqueness theory for the latter equation, translating back the results to equation \eqref{E:ipme} via \eqref{equiv}. We study such problems both in a bounded domain and in the whole space, and we establish regularity, finite propagation, and asymptotic behaviour under convenient assumptions on the data.

\subsection{\bf Results on a bounded domain}

We couple equation \eqref{E:mu} with initial and boundary data:
\begin{equation}\label{E:data_g}
	u(x,t) = g(x,t)\;\; \text{on $\Gamma$},
\end{equation}
where $\Gamma = \del_p Q$ is the parabolic boundary if our domain $Q = \Omega\times[0,T]$, $\Omega\subset \Rd$. As mentioned above, if the problem is posed in a ball with radially symmetric data it is observed that the solutions are equivalent to solutions for the $1$-$d$ porous medium equation, for which there is a weak theory that gives existence and uniqueness. Such a theory is not available for more general solutions of the fully-nonlinear equation \eqref{E:mu}. Therefore, we introduce the concept of viscosity solution. Since the equation is both singular and degenerate, we adapt the modified notion of viscosity solution introduced in \cite{CV} and \cite{BV} and the viscosity method to tackle degenerate problems. Our first main result concerns
solutions with strictly positive data.

\begin{theorem}\label{T:exist_pos}
Let $\Omega$ be an open and bounded domain in $\Rd$. For each $g\in C(\Gamma)$, $g \geq c$, there exists a
viscosity solution $u\in C(\overline{Q})$ of \eqref{E:mu} satisfying $u = g$ on $\Gamma$. Moreover, $u\geq c$
and the modulus of continuity of $u$ can be estimated in terms of the modulus of continuity of $g$; in particular, if
$g$ is Lipschitz, then so is $u$. If $g$ is Lipschitz on the lateral portion of $\Gamma$, then the solution is unique.
\end{theorem}

Our main interest lies however in treating nonnegative data that may vanish somewhere, in the spirit of
porous medium equations with finite speed of propagation. This is not so easy. If we do not assume the data to be strictly
positive, we can still obtain a {\sl maximal viscosity solution}.

\begin{theorem}\label{T:exist_maximal}
Given $g\in C(\Gamma)$ with $g\geq 0$, there exists a viscosity solution of \eqref{E:mu}, $\bar{u}\in C(\overline{Q})$,
satisfying $\bar{u} = g$ on $\Gamma$, and such that if $v$ is another viscosity solution of \eqref{E:mu} with the
same initial and boundary data, then $\bar{u} \geq v$. Moreover, if $g$ is Lipschitz continuous, then $\bar{u}$
is also Lipschitz continuous.
\end{theorem}
Maximal viscosity solutions are our choice of good class of solutions to work with.

Let us outline the organization of the proofs of this and related results. In Section~\ref{S:viscosity_solutions} we define viscosity sub- and
supersolutions for this problem and prove a useful lemma to simplify testing a solution.

In Section~\ref{S:pos_solutions} we consider strictly positive solutions. We prove two comparison results: a ``weak''
and a ``strong'' version. The strong version is valid for functions which are Lipschitz on the lateral boundary, and
is in fact just a regular comparison result.

In Section~\ref{S:special_solutions} we compute some explicit solutions. Of particular interest are the
\textit{Barenblatt functions}, which are similarity solutions, and the traveling waves.

For existence, the idea is to regularize the problem as is done in \cite{JK} for the infinity Laplacian evolution, and obtain estimates independent of the approximation process. In Section~\ref{S:maximal_solutions} we prove the existence of maximal viscosity nonnegative solutions and give some properties of these solutions, in particular we prove a uniqueness result for this problem in a nice domain (slightly more than star shaped).

Section~\ref{S:exist} is the longest and contains the more technical proof of the main result, Theorem~\ref{T:exist_pos}.

To complete the study in bounded domains, Section \ref{sec.asymp.bdd} discusses the large time behaviour of solutions of the Dirichlet problem, which is described in the following theorem.
\begin{theorem}\label{T:large_time_Dirichlet}
Let $u$ be the maximal solution of \eqref{E:mu}--\eqref{E:data_g} with $g(x,t) = 0$ for $x\in \del\Omega$. Then,
\begin{equation}
	\lim_{t\to\infty} t\,u(x,t) = U(x) = U_\Omega(x) = [F_\Omega(x)]^{m-1},
\end{equation}
where $U$ is a Lipschitz continuous function, positive in $\Omega$ which vanishes on $\del\Omega$.
\end{theorem}

As an interesting consequence of our study here, we obtain existence of solution of the elliptic eigenvalue problem
\begin{equation}
	-\lapi G = \lambda G^p, \quad p<1.
\end{equation}
As in \cite{CV} and \cite{BV}, we approximate the data with a strictly positive function and then take the limit. This allows us to obtain the asymptotic behaviour of solutions for the Dirichlet problem in Theorem~\ref{T:large_time_Dirichlet} as saying that $u(x,t)$ approaches the so-called Friendly Giant, $U(x,t)= t^{-1}F_\Omega(x)$, whose existence is well-known in the PME theory. Our results are also consistent with the asymptotic profile obtained by Lauren\c{c}ot and Stinner \cite{LS} for the
infinity heat equation. \begin{rmk}
We will show that the function $F_\Omega$ is in fact a constant multiple of the solution of the elliptic eigenvalue problem in $\Omega$ \eqref{E:eevp}.
\end{rmk}

\subsection{\bf Results on the whole space}

In Section~\ref{S:Cauchy} we discuss the Cauchy problem, posed in the whole space for nonnegative solutions. We assume that the initial function is nonnegative and bounded. We prove existence of a maximal solution. We can not obtain a complete comparison result in this case, but for compactly supported data the results for maximal solutions are valid,
namely the propagation of the free boundary and its regularity in the form of the following theorem.

\begin{theorem}\label{thm.asbeh} Let us assume  that $\rho_0:\Rd\to\R$ is continuous, nonnegative, bounded and compactly supported. If $u(x,t)$ is the maximal solution of the Cauchy problem \eqref{E:mu_cauchy}, then there exists a constant $R>0$ which depends on $m,n$ and the initial data $\rho_0$, such that as $t\to \infty $, $\rho(x,t) =
\left[ \frac{m-1}{m}u \right]^{1/(m-1)}$ satisfies
\begin{equation}
	t^{\frac{1}{m+1}}\left|\rho(x,t)-\beta_R(x,t)\right|\to 0
\end{equation}
uniformly in $x\in \RR^n$, where $\beta_R$ is the Barenblatt function defined in \eqref{E:Barenblatt_def}. Moreover, we have
convergence of the supports
\begin{equation}
B_{R_1(t)}(0) \subset \{x: \rho(x,t)>0 \}\subset B_{R_2(t)}(0)
\end{equation}
where $R_i(t)/t^{1/(m+1)}\to R=R(m,n,\rho_0)$.
\end{theorem}
The Barenblatt function $\beta_R(r,t)$ is explained in formula \eqref{E:Barenblatt_def} and $R$ is its radius at $t=1$. There is to our knowledge no simple formula to express the dependence of the asymptotic constant $R$ on the data $\rho_0$. Such a difficulty is relatively frequent in problems in nonlinear mechanics, see \cite{KPV}.

The paper ends with an appendix on special solutions.

\medskip

\subsection{\bf Extension}

\noindent A natural development of the main idea of this work leads to consider similar models of propagation in a preferential direction. One option is to use interpolation of the regular Laplacian with the infinity Laplacian. We propose as the simplest example the family of equations

\begin{equation}\label{E:interpol}
	\rho_t = L_\eps (\rho^m), \qquad \mbox{where } \quad L_\eps= \eps\,\Delta + (1-\eps)\lapi
\end{equation}
with $0<\eps<1$. Since the standard $p$-Laplacian operator, $1<p<\infty$, is defined as
\begin{equation}
	\Delta_p u= \div (|Du|^{p-2}Du)=|Du|^{p-2}\sum_{i,j}u_{x_ix_j}\left\{\delta_{ij}+ (p-2)\frac{u_{x_i}}{|D u|}\frac{u_{x_j}}{|D u|}\right\}\,,
\end{equation}
if we put $\eps=1/(p-1)$ we can also write this proposed model as
\begin{equation}
	\rho_t =\eps \,|D(\rho^m)|^{2-p}\Delta_p(\rho^m).
\end{equation}
Many of the results of the present paper apply, at least partially, to the interpolated family. In particular, we note that, at least for radial solutions, \eqref{E:interpol} is like a (PME) in dimension $1+\eps(d-1)$. We will not pursue such an analysis in the present work.

\medskip

\noindent{\bf Notations}

We consider equations defined on some subdomain of the whole
Euclidean space-time $S = \Rd\times\R$. For a point $P_0 = (x_0,t) \in S$, we say that $U$ is
a parabolic neighbourhood of $P_0$ if $P_0\in Q\subset U$, where $Q$ is a cylinder centered at
$P_0$, that is, $Q = B_r(x_0)\times(t_0 - \tau,t_0]$ for some $r,t>0$. We denote by $C^+(D)$
the space of nonnegative continuous functions from $D$ to $\R$, whereas $C^{2,1}(D)$ will
denote those functions which are twice differentiable in $x$ and once in $t$. Whenever $Q$ is a
cylinder of the form $Q = \Omega\times[0,T]$, with $\Omega \subset \Rd$ open, we denote its
parabolic neighborhood by $\Gamma = \del_p Q = (\del \Omega\times[0,T])\cup(\Omega\times\{t=0\})$.

For a function $u\in C^{2,1}(D)$, $Du$ and $D^2u$ will denote the spatial gradient and the $d\times d$
matrix of second derivatives of $u$, respectively,
\[
	Du = \left( u_{x_1},\ldots, u_{x_d} \right), \;\;\;
	D^2u = \left(\left(u_{x_ix_j}\right)\right)_{i,j=1}^d.
\]

Given a symmetric $d\times d$ matrix with real coefficients $A$, we define
\[
	\Lambda(A) \bydef \max_{\omega\in S^{d-1}} (A\omega)\cdot\omega, \;\;\;
	\lambda(A) \bydef \min_{\omega\in S^{d-1}} (A\omega)\cdot\omega,
\]
in other words, the largest and smallest eigenvalues of $A$, respectively. We will also need an approximation function, $\beta_c$, which is a smooth real function satisfying $\beta_c(z) = |z|$ if $|z|\geq c$ and $\beta_c(z) \geq c/2$ everywhere.

\section{Viscosity solutions}\label{S:viscosity_solutions}

Following \cite{CV} and \cite{BV}, we define viscosity solutions for the modified equation \eqref{E:mu}. Note
that because the equation is singular at points where the gradient of the function vanishes, the
usual definition of viscosity solution needs to be adapted at the singular points. We adapt the
definition from \cite{CGG}, see also \cite{JK}.

\begin{defn}\label{D:visc_subsol}
Given $u\in C^+(\overline{Q})$, we say $u$ is a \textbf{nonnegative viscosity subsolution} of \eqref{E:mu}
in $Q$ if and only if for every $P_0\in Q$ and every function $\phi\in C^{2,1}(Q)$ which touches
$u$ from above at $P_0$, the following holds at the point $P_0$:
\[\begin{aligned}
	& \phi_t \leq k\,\phi\,\lapi\phi + |D\phi|^2 \qquad&\text{if }D\phi\neq 0,\\
	& \phi_t \leq k\,\phi\, \Lambda(D^2\phi) \qquad&\text{if }D\phi = 0.
\end{aligned}\]
\end{defn}

As in \cite{JK} we can weaken the second condition.
\begin{lemma}\label{L:weak_sol_cond}
Let the function $u$ satisfy the following: given $P_0\in Q$ and $\phi\in C^{2,1}(Q)$ such that
$u - \phi$ has an absolute maximum at the point $P_0$ and $u(P_0) = \phi(P_0)$, it follows that
at the point $P_0$
\[\begin{aligned}
	& \phi_t \leq k\,\phi\,\lapi\phi + |D\phi|^2 \qquad&\text{if }D\phi\neq 0,\\
	& \phi_t \leq 0 \qquad&\text{if }D\phi = 0,\;D^2\phi = 0.
\end{aligned}\]
Then $u$ is a viscosity subsolution of \eqref{E:mu}.
\end{lemma}

\begin{proof}
Step 1. Assume $u$ is not a viscosity subsolution of \eqref{E:mu} but satisfies the condition of
the lemma. Then there exist $P_0 = (x_0,t_0) \in Q$ and $\phi\in C^{2,1}(Q)$ such that $u-\phi$
has an absolute maximum at $P_0$, $u(P_0) = \phi(P_0)$, $D\phi(P_0) = 0$, $D^2\phi(P_0) \neq 0$,
and at $P_0$
\[
	\phi_t > k\, \phi\, \Lambda(D^2\phi).
\]
Let us define
\[
	w_j(x,t,y,s) \bydef u(x,t) - \phi(y,s) - \frac{j}{4}|x-y|^4 - \frac{j}{2}|t-s|^2
\]
and let $(x_j,t_j,y_j,s_j)$ be a point of maximum for $w_j$ in $\overline{Q}\times\overline{Q}$. It is
easy to see that $(x_j,t_j,y_j,s_j) \to (x_0,t_0,x_0,t_0)$ as $j\to\infty$.

Step 2. Let us check that for $j$ large enough we can not have $x_j = y_j$. Assume $x_j = y_j$ and let
us define a new function
\[
	\theta_j(y,s) \bydef -\frac{j}{4}|x_j-y|^4-\frac{j}{2}|t_j-s|^2 
\]
Then $\phi - \theta_j$ has a local minimum at the point $(y_j,s_j)$ and hence, at this point
$\phi_t = \theta_{j,t}$ and $D^2\phi \geq D^2\theta_{j} = 0$. From our assumption, for $j$ large
enough we have
\[
	j(t_j-s_j) = \theta_{j,t}(y_j,s_j)= \phi_t(y_j,s_j) > k\,\phi(y_j,s_j)\Lambda(D^2\phi(y_j,s_j)) \geq 0.
\]
Similarly, with
\[
	\zeta_j(x,t) \bydef \frac{j}{4}|x-y_j|^4 + \frac{j}{2}|t-s_j|^2 
\]
the function $u-\zeta_j$ has a maximum at $(x_j,t_j)$ and both $D\zeta_j$ and $D^2\zeta_j$ vanish
at $(x_j, t_j)$, because $x_j=y_j$. Therefore, $\zeta_j + C$ satisfies the condition of the lemma
and we conclude that
\[
	j(t_j-s_j) = \zeta_{j,t}(x_j,t_j) \leq 0,
\]
a contradiction.

Step 3. We have concluded that $x_j\neq y_j$ for large $j$. Let us now check that this also leads to a
contradiction.
By our assumption, and using the continuity of $P \mapsto \Lambda(D^2\phi(P))$, there exists
$\eps > 0$ such that in a neighbourhood of $(x_0,t_0)$
\[
	\eps < \phi_t - k\,\phi\,\Lambda(D^2\phi) - |D\phi|^2.
\]
Therefore, for large $j$, using the fact that $\phi-\theta_j$ has a minimum at $(y_j,s_j)$,
\[\begin{aligned}
	\eps &< \left( \phi_t - k\,\phi\,\lapi\phi - |D\phi|^2\right)(y_j,s_j) \\
	&\leq j(t_j-s_j) - k\,\frac{\phi(y_j,s_j)}{|x_j-y_j|^2}\langle D^2\theta_j(y_j,s_j)
		\cdot(x_j-y_j), (x_j-y_j) \rangle \\
	&\qquad- j^2|x_j-y_j|^6 .
\end{aligned}\]
On the other hand, we can still apply the condition of the lemma to $\zeta_j$, but now
$D\zeta_j \neq 0$. Using this and the fact that $u-\zeta_j$ has a maximum at $(x_j,t_j)$ we have
\[
	j(t_j-s_j) - j^2|x_j-y_j|^6 \leq k\frac{u(x_j,t_j)}{|x_j-y_j|^2}\langle D^2\zeta_j(x_j,t_j)
		\cdot(x_j-y_j),(x_j-y_j) \rangle.
\]
Since $D^2\zeta_j(x_j,t_j) = - D^2\theta_j(y_j,s_j)$, this and the above equation lead to a
contradiction for large $j$.
\end{proof}

\begin{defn}\label{D:cfbs}
Given $u\in C^+(\overline{Q})$, we say $u$ is a \textbf{nonnegative classical free-boundary solution} of
\eqref{E:mu} if and only if:
\begin{itemize}
\item[(i)] on the positivity set $\mathcal{P}(u) = \{ P\in Q \mid u(P) > 0\}$, the function $u$ is
smooth and solves \eqref{E:mu} in the classical sense;
\item[(ii)] the boundary of the positivity set $\Gamma = \del\mathcal{P}(u)\cap Q$ is a smooth
hypersurface and $u\in C^{2,1}(\mathcal{P}(u)\cap\Gamma)$;
\item[(iii)] on the hypersurface $\Gamma$ we have
\begin{equation}\label{E:bdry_speed}
	\sigma_n = |Du|,
\end{equation}
where $\sigma_n$ denotes the normal speed of boundary $\Gamma$.
\end{itemize}
\end{defn}

If instead of condition (\textit{i}) above we have
\[
	u_t \leq k\,u\,\lapi u + |Du|^2,
\]
in the positivity set and instead of \eqref{E:bdry_speed} we have
\[
	\sigma_n \leq |Du|,
\]
then we say that $u$ is a \textit{classical free-boundary \textbf{subsolution}}.
If in Definition~\ref{D:cfbs} we impose the extra condition
\begin{itemize}
\item[(\textit{iv})] $|Du| \neq 0$ on $\Gamma$,
\end{itemize}
then we say that $u$ is a \textit{classical \textbf{moving} free-boundary solution}. Replacing
all the inequalities with $\leq$ by inequalities with $\geq$ we define \textit{classical
(moving) free-boundary \textbf{supersolution}}.

Given two functions $u,v\in C^+(D)$, we say that $u$ is \textbf{strictly separated} (from above)
from $v$ if $u$ is compactly supported and $u(x) < v(x)$ for every $x\in supp(u)$, in this case
we write $u\prec v$.

\begin{defn}\label{D:visc_supsol}
Given $u\in C^+(\overline{Q})$, we say $u$ is a \textbf{nonnegative viscosity supersolution} of \eqref{E:mu}
if and only if the following conditions are satisfied:
\begin{itemize}
\item[(i)] For every $P_0\in Q$ where $u(P_0) > 0$ and every function $\phi\in C^{2,1}(Q)$
which touches $u$ from below at $P_0$ we have at the point $P_0$
\[\begin{aligned}
	& \phi_t \geq k\,\phi\,\lapi\phi + |D\phi|^2 \qquad&\text{if }D\phi\neq 0,\\
	& \phi_t \geq k\,\phi\, \lambda(D^2\phi) \qquad&\text{if }D\phi = 0.
\end{aligned}\]
\item[(ii)] If $w$ is a classical moving free-boundary subsolution of \eqref{E:mu} which
is strictly separated from $u$ at time $t_1$, $w(\cdot,t_1)\prec u(\cdot,t_1)$ and which satisfies $w < u$ on
$\Gamma_{t_1,t_2} = \{ (x,t) \in \del_p Q \mid t_1 \leq t \leq t_2\}$, the portion of lateral boundary of $Q$
for times between $t_1$ and $t_2$  (when the space domain is not the whole $\Rd$), then $w$ can not cross $u$
for times in $[t_1,t_2]$, \textit{i.e.} $w(\cdot,t) \leq u(\cdot,t)$ for every $t\in[t_1,t_2]$.
\end{itemize}
\end{defn}
See \cite{BV} for a motivation of these definitions. The main idea of comparing with free-boundary
solutions goes back to \cite{CV}. Finally we can define viscosity solutions.

\begin{defn}
A function in $C^+(\overline{Q})$ is a nonnegative \textbf{viscosity solution} of \eqref{E:mu} if it is simultaneously
a viscosity subsolution and a viscosity supersolution.
\end{defn}

\subsection*{Other ways to define viscosity solutions}

It is possible to define viscosity solutions directly for equation \eqref{E:ipme}
\begin{equation}\label{E:ipme_bis}
	\rho_t = \lapi \rho^m
\end{equation}
or, taking $w = \rho^m$ and $\beta(z) = z^{\oo{m}}$, for the alternate equation
\begin{equation}\label{E:ipme_alt}
	\beta(w)_t = \lapi w.
\end{equation}
This last formulation, in particular, has exactly the same kind of degeneracy as the equation \eqref{E:mu} we
use to define viscosity solutions and obtain existence. However, the ``Barenblatt'' solutions for these equations
would not have a nonzero normal boundary derivative, and hence they seem to be less natural regarding the free
boundary propagation. It is interesting to note, however that using the transformations $u = \frac{m}{m-1}\rho^{m-1}$ and $w=
\rho^m$ on test functions, we can go back and forth from the definition of viscosity solution of one formulation
to the other. In fact, without having defined it, we will use this ``jumping'' between formulations in the proof
of Theorem~\ref{T:exist_eevp}.

\section{Strictly positive solutions on bounded domains}\label{S:pos_solutions}

To obtain maximal viscosity solution for \eqref{E:mu}, we have to approximate the solution from
above by positive solutions. Hence, we need to establish comparison and obtain estimates for positive
solutions. In what follows, $Q$ denotes a cylinder of the form $\Omega\times(0,T)$, where $\Omega$ is an open
and bounded domain in $\Rd$, and $\Gamma$ its parabolic boundary,
$\Gamma = (\Omega\times\{0\})\cup (\del \Omega\times[0,T])$.

If a function $u$ is a strictly positive classical solution of \eqref{E:mu} in $Q$, in the sense that there
exists $c$ such that
\[
	u(x) \geq c > 0 \quad \text{for every $x\in Q$},
\]
then $u$ is also a viscosity solution
(with obvious adaptations in the above definitions) of the equation
\begin{equation}\label{E:mund}
	u_t = k\,\beta_c(u) \lapi u + |Du|^2.
\end{equation}
where $\beta_c$ is as defined in the notations at the end of the introduction. Therefore, it is enough to establish the results for such functions.

\subsection{Comparison} We prove two comparison results. The proof of the first comparison result is typical and
its idea is similar to the proof of Lemma~\ref{L:weak_sol_cond}, however it is weaker than the second. We sketch
it here for convenience.
\begin{theorem}\label{T:comparison_pos}
Assume $u$ and $v$ are viscosity sub- and supersolutions of \eqref{E:mund}, respectively, in
$Q$ and satisfy
$u,v\geq c$,
\[
	 u(P) < v(P)
\]
for every $P \in \Gamma$. Then $u \leq v$ in $Q$. If $v$ is a strict supersolution or $u$ is a strict
subsolution, then this inequality is strict.
\end{theorem}
\begin{proof}
Step 1. Let us assume the statement of the theorem is false,
\[
	\sup_{P\in Q} (u(P) - v(P)) > 0.
\]
If we consider instead of $v$ the function $\bar{v}(x,t) = e^{\gamma t}v(x,h(t))$, where $h(t) =
\frac{e^{\gamma t} -1}{\gamma}$, with $\gamma$ sufficiently small we must still have $\bar{v} > u$ on
$\del_p (\Omega\times[0,T-\eps(\gamma)))$ (where $\eps(\gamma)$ is a small number depending on $\gamma$) and
now $\bar{v}$ is a strict supersolution. To see this, assuming $v$ is smooth, we compute
\[\begin{aligned}
	\bar{v}_t &= \gamma \bar{v} + e^{2\gamma t} v_t(x,\tau)
	\geq \gamma\bar{v} + k \bar{v} \lapi \bar{v} + |D\bar{v}|^2.
\end{aligned}\]
This computation can be carried for the test functions in the definition of viscosity solution, whence
our claim follows. Therefore we assume $v$ is a strict supersolution.
By considering a shorter time interval we can also assume
\[
	\sup_{P\in Q} (u(P) - v(P)) = 0.
\]
This supremum must occur at an interior point, $P_0 = (x_0,t_0) \in Q$.

As in the proof of Lemma~\ref{L:weak_sol_cond} we take
\[
	w_j(x,t,y,s) \bydef u(x,t) - v(y,s) - \frac{j}{4}|x-y|^4 - \frac{j}{2}(t-s)^2
\]
and let $\mathbf{P}_j = (x_j,t_j,y_j,s_j)$ be a point of maximum of $w_j$ in $\overline Q \times
\overline Q$. Once again, for $j$ large enough, $x_j,y_j\in \Omega$, $t_j,s_j \in (0,T)$ and
$\mathbf{P}_j \to (x_0,t_0,x_0,t_0)$ as $j\to\infty$.

Step 2. If $x_j = y_j$, then, with $\theta_j$ as in Lemma~\ref{L:weak_sol_cond}, $(v - \theta_j)(y,s)$ has
a local minimum at the point $(y_j,s_j)$, and since $v$ is a strict supersolution we get $\delta
< j(t_j - s_j)$. Similarly, $u - \zeta_j$ has a local maximum at $(x_j,t_j)$ from where we get the
contradiction $0 \geq j(t_j-s_j)$.

Step 3. Since $u(P_0) = v(P_0)$, the case $x_j \neq y_j$ can be treated exactly as in the proof of
Lemma~\ref{L:weak_sol_cond}.

If either $u$ is a strict subsolution or $v$ is a strict supersolution, then we can obtain the
contradiction without rescaling the function, and hence, without taking the limit of the rescaled solution
and loosing the strict inequality.
\end{proof}

\begin{corl} If $u$ is a subsolution and $u \leq M$ on $\Gamma$, then $u\leq M$ in $Q$. Similarly, if
$u$ is a supersolution and $u \geq c > 0$ on $\Gamma$, then $u \geq c$ in $Q$.
\end{corl}
\begin{proof}
These facts follow from the theorem noting that $M + \eps$ and $c - \eps$ are positive solutions of the
equation for sufficiently small $\eps$.
\end{proof}

The comparison theorem \ref{T:comparison_pos} is too weak to prove uniqueness because we assume a strict relation
$v>u$ on the boundary. However, the transformation we used in its proof to obtain a strict supersolution $\bar{v}$
above the subsolution $u$ will work with an extra Lipschitz condition on the data.

\begin{theorem}\label{T:comparison_pos_strong}
Assume $u$ and $v$ are viscosity sub- and supersolutions of \eqref{E:mund} in $Q$, respectively, they satisfy
$u,v\geq c > 0$, and
\[
	 u(P) \leq v(P)
\]
for every $P \in \Gamma$. Assume also that either $v_t$ is bounded below or $u_t$ is bounded above on the lateral
boundary of $Q$. Then $u \leq v$ in $Q$.
\end{theorem}
\begin{proof}
Step 1. The proof is the same as for Theorem~\ref{T:comparison_pos} once we check that for $\gamma$ sufficiently
small the function
\[
	\bar{v}(x,t) = h'(t) v\left(x,h(t)\right), \;\text{with}\;\, h(t) = \frac{e^{\gamma t+\gamma^3} - 1}{\gamma},
\]
stays strictly above $u$ on the parabolic boundary of $Q$. This is clear on the bottom portion; we claim this is
also the case on the lateral portion of the boundary. We give the proof under the assumption that $v_t$ is bounded below.
The argument when $u_t$ is bounded above is completely symmetric.

Step 2. Let us assume for the moment that the following condition holds,
\begin{equation}\label{E:bdry_grow_cond}
	v_t > - \frac{2v}{t}.
\end{equation}
Let us fix $x\in\del \Omega$ and write $g(t) = v(x,t)$. Suppose that for some $t>0$
\[
	\bar{v}(x,t) = h'(t)g(h(t)) \leq g(t) = v(x,t).
\]
Noting that
\[
	h'(t) - 1 = e^{\gamma t + \gamma^3} - 1 = \gamma t + O(\gamma^2)
\]
and
\[
	h(t) - t = \frac{\gamma t^2}{2} + O(\gamma^2),
\]
we must have
\[\begin{aligned}
	g'(t) &= \frac{g(h(t)) - g(t)}{h(t) - t} + o(1) \leq -\frac{h'(t) - 1}{h(t) - t}g(t) + o(1)
	= - \frac{2g(t)}{t} + o(1).
\end{aligned}\]
This is in contradiction with \eqref{E:bdry_grow_cond} for $\gamma$ sufficiently small, therefore we must have
$\bar{v}(x,t) > v(x,t)$ as claimed.

Step 3. Finally, we remove the condition \eqref{E:bdry_grow_cond}. Since we are assuming $v_t$ bounded below, and $v\geq c >0$,
this condition must surely hold for small $t$, hence, $v \geq u$ on $\Omega\times[0,\bar{t}]$ for small $\bar{t}$. Since
the equation is invariant under translations in $t$, we can now reapply the result to the intervals of the form
$[k\bar{t},(k+1)\bar{t}]$, $k=1,2,3,\ldots$, successively,
considering $\tilde{u}(x,t) = u(x,t-k\bar{t})$ and $\tilde{v}(x,t) = v(x,t-k\bar{t})$.
\end{proof}

\section{Special solutions}\label{S:special_solutions}

We look briefly at concrete examples of solutions with a free boundary which are smooth on their set
of positivity. All the solutions obtained in this section are classical moving free-boundary solutions in the sense of Definition~\ref{D:cfbs}. These solutions are important as particular examples, and also as
models of asymptotic behavior of general classes of solutions as $t\to\infty$.

\subsection{Separation of variables}
Solutions of the form
\[
	\rho(x,t) = T(t) F(x)
\]
are usually of interest. Using the equation, we get the explicit formula for $T$
\begin{equation}\label{E:sep_T}
	T(t) = \left[ C + (m-1)\lambda t \right]^{-\oo{m-1}}
\end{equation}
and the equation for $F$
\begin{equation}\label{E:sep_F}
	\lapi F^m(x) + \lambda F(x) = 0.
\end{equation}

A special case of separation of variables is when $\rho$ is independent of $t$, which corresponds to
taking $\lambda = 0$. In this case the equation is nothing more than the infinity Laplace equation. Thus,
if $v$ is any nonnegative solution of $\lapi v = 0$ in $\Omega$, then $\rho(x,t) \bydef v^{\oo{m}}(x)$ is a
solution of \eqref{E:ipme} in $Q$.

Solving the nonlinear elliptic equation \eqref{E:sep_F} is nontrivial. This can nonetheless be done explicitly
when $F$ is radial (we show there is a solution in a non-radial domain in Theorem~\ref{T:exist_eevp}). Since
the solutions obtained this way are known to exist, we postpone their derivation to Appendix~A (although their
explicit form seems to be new). As observed in the Introduction, the radial solutions are essentially solutions
of the 1-$d$ (PME). In fact, all the explicit solutions we obtain are of this kind.

\subsection{Similarity solutions}

If we look for nonnegative solutions of \eqref{E:ipme} of the form
\[
	\rho(x,t) = t^{-\alpha} f(\eta),
\]
with $\eta = t^{-\beta}x$ and further assume that $f$ is radial and $\alpha = \beta = 1/(m+1)$ we obtain the
following solution:
\[
	f(r) = \left[ A - \frac{\beta (m-1)}{2m} r^2\right]_+^{\oo{m-1}}.
\]
By analogy with the porous medium equation, we call this type of solution the \textit{Barenblatt functions}. The
solution $\rho$ takes the form
\begin{equation}\label{E:Barenblatt_def}
	 \beta_{R}(x,t) \bydef \frac{\gamma_m}{t^\oo{m-1}} \left[ \left(Rt^\oo{m+1}\right)^2
		- |x|^2 \right]_+^\oo{m-1},
\end{equation}
where $\gamma_m = \left[ (m-1)/(2m(m+1)) \right]^\oo{m-1}$. Observe that $R$ denotes the radius of the
support of $\beta_R$ at time $t=1$. As observed above, these Barenblatt functions are closely related to the
Barenblatt functions in \cite[formula (4.2)]{BV}, more precisely, they are the $(m-1)$-th power of our functions
in dimension 1, which is consistent with the intuition that $\lapi$ is a one dimensional second derivative.

We will need these solutions for the modified problem, that is, we want to apply the pressure-to-density
transformation $u = (m/(m-1)) \rho^{m-1}$. Therefore, we get the \textbf{Barenblatt functions} for \eqref{E:mu}
\[
	\mathcal{B}_R(x,t) = \oo{(m+1) t} \left[ \left(Rt^\oo{m+1}\right)^2 - |x|^2 \right]_+.
\]
It will, of course, be convenient to consider translations of these functions, in particular in $t$, to
avoid a singularity at $t=0$: for $(x_0,t_0) \in \Rd\times\R^+_0$, define
\[
	\mathcal{B}_R(x,t;x_0,t_0) \bydef \mathcal{B}_R(x-x_0,t+t_0)
\]

\begin{rmk}
The choice $\alpha = \beta$ is the only one that makes $\int_\Rd [\rho(x,t)]^d dx$ constant in time. This
is in some sense the ``correct'' behaviour for the infinity-Laplacian, as the ``fundamental'' solution
of $u_t = \lapi u$, $\Gamma(x,t) = t^{-\oo{2}} \exp(-|x|^2/4t)$ satisfies this type of conservation.
\end{rmk}

\subsection{Traveling waves}

For functions which depend only on one space variable, the infinity-Laplacian is just a regular second
derivative. Therefore, if we take $\rho$ of the form $\rho(x,t) = f(\eta)$, $\eta = x_1-ct$, with $c>0$,
we have
\[
	(f^m){''} + cf{'} = 0.
\]
This is the exact same equation obtained, under similar assumptions, for the regular porous medium equation
(PME), $u_t = \lap (u^m)$, so we get the same solutions, namely,
\[
	\rho(x,t) = c[a+ct-x_1]_+.
\]
The traveling wave solution for \eqref{E:mu} is therefore given by
\[
	u(x,t) = \frac{m}{m-1}c[a+ct-x_1]_+^{m-1}.
\]

For more general planar waves of the form $\rho(x,t) = A(t) U(x_1-s(t))$, the equation for $A$ and $U$ is
again the same as for the (PME), and we still have the same solutions.

It is easy to see that both the Barenblatt functions and the traveling waves satisfy the definition of
classical free boundary solutions.

\section{Maximal viscosity solutions in a  bounded domain}\label{S:maximal_solutions}

In this section we prove existence of the maximal viscosity solution of \eqref{E:mu} for nonnegative data as stated in Theorem~\ref{T:exist_maximal}. The idea is to approximate the data with a strictly
positive function.

\begin{proof}[Proof of Theorem~\ref{T:exist_maximal}]
Step 1. For $n\in\mathbb{N}$ let $g^n = g + \oo{n}$ and let $u^n$ be the unique viscosity solution of
\eqref{E:mu}, satisfying $u^n = g^n$ on $\Gamma$, given by Theorem~\ref{T:exist_pos}. From the
comparison results, we have that $0 < u^n \leq u^l$ in $Q$ for $l < n$. Hence, $(u^n)_{n\in\mathbb{N}}$
has a limit as $n \to\infty$, $\bar{u}$. From the estimate of Theorem~\ref{T:exist_pos} we immediately see that $\bar{u}$
is continuous---if $g$ is Lipschitz, we see that $\bar{u}$ is also Lipschitz continuous.
Note also that because $\bar{u}$ is continuous and $(u^n)_{n}$ is decreasing, the convergence is locally uniform.

Step 2. Next we check that $\bar{u}$ is a viscosity subsolution of \eqref{E:mu}. The argument goes exactly as in
\cite[Proposition~5.1]{BV} for the nonsingular part of the inequality in Definition~\ref{D:visc_subsol}. Take a
smooth function $\phi$ touching $\bar{u}$ from above at a point $P_0 = (x_0,t_0)$ and assume first that $D\phi(P_0)
\neq 0$. Since $u^n$ converges locally uniformly to $\bar{u}$, in some neighborhood of $P_0$, $Q_r = B_r(x_0)\times
(t_0 - r^2, t_0]$, for large $n$ the function $\phi^n = \phi + c_n$ touches $u^n$ from above at a
point $P_n \to P_0$ as $n\to\infty$. Note that $c_n$ is chosen so that the maximum of $u^n-\phi^n$
(at $P_n$) is exactly 0 and for $n$ sufficiently large $D\phi^n \neq 0$ in $Q_r$. It is clear that
$c_n\to 0$ as $n\to\infty$, and since $u^n$ is a positive viscosity solution,
\[
	\phi^n_t \leq k \phi^n\lapi\phi^n + |D\phi^n|^2 \qquad\text{at $P_n$},
\]
letting $n\to \infty$ we obtain the required inequality for $\phi$,
\[
	\phi_t \leq k \phi \lapi \phi + |D\phi|^2 \qquad\text{at $P_0$}.\]
If $D\phi(P_0) = 0$, from Lemma~\ref{L:weak_sol_cond} we can also assume $D^2\phi(P_0) = 0$ and need only check that
$\phi_t \leq 0$. With a similarly constructed family $(\phi^n)_{n\in\mathbb{N}}$, we have that at the point $P_n$
either
\[\begin{aligned}
	&\phi^n_t \leq k\phi^n \lapi \phi^n + |D\phi^n|^2, &\;\;
		&\text{if $D\phi^n(P_n)\neq 0$, or} \\
	&\phi^n_t \leq k\phi^n\Lambda(D^2\phi^n), & \;\; &\text{if $D\phi^n(P_n) = 0$.}
\end{aligned}\]
In any case, letting $n\to \infty$, we get
\[
	\phi_t \leq 0,
\]
as required.

Step 3. To prove that $\bar{u}$ is a viscosity supersolution we need to check conditions $(i)$ and $(ii)$ of
Definition~\ref{D:visc_supsol}. Condition $(i)$ follows just as in step 2 of this proof. Let us check the second
condition. Assume $v$ is a classical moving free boundary subsolution of \eqref{E:mu} strictly separated from
$\bar{u}$ at $t = t_1$ and on the portion of the lateral boundary $\Gamma_{t_1,t_2}$.
From our construction, we have $\bar{u} \leq u^n$ in $Q$, in particular $v < u^n$ at
$t = t_1$ and on $\Gamma_{t_1,t_2}$. Since $v$ is classical at points where it is positive, we obtain from the
comparison result for positive solutions that $v$ can not touch $u^n$ from below. Passing to the limit we get that
$v$ can not cross $\bar{u}$.

Step 4. Finally, let us check that $\bar{u}$ is maximal. Suppose $w$ is another viscosity solution with the same data
as $\bar{u}$ which is strictly greater than $\bar{u}$ at some point. For $n$ sufficiently large, $w$ touches
$u^n$ from below for the first time at some point $P_n = (x_n,t_n)$ with $t_n >0$. Since $w = \bar{u} = g$ on
$\Gamma$ and $u^n = g + \oo{n}$ on this set, $P_n$ must be an interior point. Furthermore, since $w(P_n) =
u^n(P_n) > 0$, it is possible to find $Q_{n} \subset  \Omega\times[t_n-\delta_n,t_n]$ such that $P_n\in Q_n$, $w > 0$
in $Q_n$ and $w<u^n$ on $\del_p Q_n$. By the comparison result, we get a contradiction. To see that $Q_n$ must
exist with the desired properties, note that since $w=g= u^n - \oo{n}$ on $\Gamma$ and $w(P_n) = u^n(P_n) > 0$,
the open set $W_n = \{x\in \Omega \mid w(x,t_n) > \min w(P_n)/2\}$ contains the point $P_n$ and $w = w(P_n)/2$ on
$\del W_n$. By continuity, there exists $\delta_n >0$ such that $w > w(P_n)/4$ on $W_n\times [t_n-\delta_n,t_n]$.
This set satisfies the above requirements.
\end{proof}

Maximal solutions are ordered according to their data.
\begin{theorem}
If $g_1$ and $g_2$ are functions in $C(\Gamma)$ and $g_1 \leq g_2$, then the maximal solutions $u_1$ and $u_2$,
obtained by the above process with data $g_1$ and $g_2$, respectively, satisfy $u_1 \leq u_2$ in $Q$. In particular,
maximal solutions are unique.
\end{theorem}
\begin{proof}
The approximations $g_1^n$ and $g^n_2$ from the previous proof satisfy $g^n_1 \leq g^n_2$, therefore, from the
strong comparison result we have $u^n_1 \leq u^n_2$, and in the limit $u_1 \leq u_2$.
\end{proof}

It is worth noting that in Step 3 of the proof of existence of a maximal solution, we in fact prove that a maximal
solution which is strictly separated at some time from a solution below can not be crossed by this solution at
later times.
\begin{theorem}
Let $u$ be a maximal solution and $v$ another solution such that $v \prec u$ at $t=t_1$. Then $v \leq u$ for all $t>t_1$.
\end{theorem}
Here $u\prec v$ means $u$ is strictly separated from $v$ as defined in Section~\ref{S:viscosity_solutions}.
The following facts are useful and easy to check.
\begin{theorem}
The Barenblatt functions and the traveling wave solutions of the previous section are maximal solutions (in the whole space)
of \eqref{E:mu}.
\end{theorem}
\begin{proof}
We check this for the Barenblatt functions of the form $\mathcal{B}_R(x,t+t_0)$ with $t_0>0$. Take $r\gg R$ and consider
the approximation given in the proof of Theorem~\ref{T:exist_maximal} at the beginning of this section, $u^n$ with initial
condition $g^n = \max\{1/n,\,\mathcal{B}(\cdot,t_0)\}$ on the cylinder $Q_r = B(0,r)\times[0,t]$. In radial coordinates,
the equation is simply the regular (PME) in one spacial dimension, therefore, as was proved in \cite{BV}, as $n\to\infty$,
the solution converges to the Barenblatt function in $Q_r$. A similar argument works for the traveling wave solutions.
\end{proof}

It is worth recording the following fact about the independence of the maximal solution on the approximation
procedure to obtain it.
\begin{theorem}\label{T:independent_approx}
Let $\{h^n \mid n\geq1\}$ be a family of functions defined on $\Gamma$ satisfying $h^{n} \geq h^{n+1} > 0$, $h^n$ is
continuous on $\Gamma$ and Lipschitz continuous on the vertical portion of $\Gamma$ for each $n\geq1$ and $\lim_n h_n = g$.
Let $v^n$ be the unique solution of \eqref{E:mu}. Then $\lim_n v^n$ is the maximal viscosity solution from
Theorem~\ref{T:exist_maximal}.
\end{theorem}
\begin{proof}
The proof of Theorem~\ref{T:exist_maximal} can be reproduced to show that $v = \lim_n v^n$ is a viscosity solution of
\eqref{E:mu}. It is therefore enough to observe that for any given $n\geq1$ there exists $m>n$ such that $g^m < h^n$ on
$\Gamma$ (where $g^n$ is as in the proof of Theorem~\ref{T:exist_maximal}) and use the comparison result to conclude that
$u^m \leq v^n$. Now take the limit in $m$ and in $n$ to see that $u \leq v$.
\end{proof}

\noindent {\sc Uniqueness of viscosity solutions.} If the domain $\Omega$ is sufficiently nice and the normal derivative on $\del\Omega$ does not vanish, then we can obtain
a uniqueness result. We will assume $\Omega$ satisfies the following \textit{contracting property}, which is slightly
stronger than being star shaped: there exist $x_0\in\Omega$, $\lambda_* <1$ and $\alpha>0$ such that
\begin{equation}\label{P:contracting_prop}
	\text{for every $x\in\del\Omega$ and every $\lambda\in[\lambda_*,1)$,} \qquad
		d(x_0 +\lambda (x-x_0),\del\Omega) \geq \alpha \lambda. \tag{CP}
\end{equation}
\begin{theorem}
Assume $\del\Omega\in C^1$ and $\Omega$ satisfies the above contracting property \eqref{P:contracting_prop}.
Assume further that $u_0\in W^{1,\infty}(\Omega)$ and that the normal derivative of $u_0$ on $\del\Omega$ satisfies
\[
	\left.\frac{\del u_0}{\del\nu}\right|_{\del\Omega} \leq -m < 0.
\]
Then the maximal solution of the Dirichlet problem \eqref{E:DP} is the unique viscosity solution of this problem.
\end{theorem}
\begin{proof}
Suppose $u$ is the maximal viscosity solution and $v$ is another solution of \eqref{E:DP}. After a translation we can
assume $x_0=0$. For
$\lambda\in[\lambda_*,1)$ define $u^\lambda:\Omega_\lambda \to\R$,
\[
	u^\lambda(x,t) \bydef \lambda^{2+\gamma} u\left(\frac{x}{\lambda},\lambda^\gamma t\right),
\]
where $\Omega_\lambda \bydef \lambda \Omega = \{\lambda x \mid x\in\Omega \}$ and $\gamma$ will be chosen below.
Likewise define $u_0^\lambda(x) = \lambda^{2+\gamma}u_0(x/\lambda)$. The function $u^\lambda$ is a viscosity solution of the
equation in $\Omega_\lambda \subset \Omega$ with initial data $u_0^\lambda$. From the assumptions on $\Omega$ and
$Du_0$ we see that choosing $\tilde{\lambda} \geq\lambda_*$ sufficiently close to $1$ and for $x\in\Omega\setminus
\Omega_{\tilde{\lambda}}$ we have
\[
	\left.\frac{d}{d\lambda} u^\lambda_0(x) \right|_{\lambda=1^-} = (2+\gamma)u_0(x) - Du_0(x)\cdot x \geq m\alpha >0.
\]
On the other hand, for $x\in\Omega_{\tilde{\lambda}}$ we can choose $\gamma$ independent of $\lambda>\tilde{\lambda}$ such that
\[
	\left.\frac{d}{d\lambda} u^\lambda_0(x) \right|_{\lambda=1^-} = (2+\gamma)u_0(x) - Du_0(x)\cdot x \geq
		(2+\gamma) \min_{\Omega_{\tilde{\lambda}}} u_0 - LR > 0,
\]
where $L$ is the Lipschitz constant of $u_0$ and $R$ is the radius of a ball containing $\Omega$. Hence, we conclude that
for $\lambda\in(\tilde{\lambda},1)$ the function $u^\lambda_0$ is strictly below $u_0$ in $\Omega_\lambda$.

We can apply the comparison result in $Q_\lambda \bydef \Omega_\lambda\times[0,T]$ to conclude that $u^\lambda < v$ in
this set. Letting $\lambda\to 1^-$ we obtain $u \leq v$. Since $u$ is maximal, this means that $u=v$.
\end{proof}

\section{Existence and regularity. Proof of Theorem~\ref{T:exist_pos}}\label{S:exist}

The existence portion of Theorem~\ref{T:exist_pos} is obtained, with the necessary adaptations, following
the approximation procedure introduced in \cite{JK}.
Define, for $\eps \geq 0$ and $\delta >0$,
\[
	\mathcal{L}^{\eps,\delta} u = \eps\lap u + \frac{k\beta_c(u)}{|Du|^2 + \delta^2}
	\langle D^2 u \cdot Du, Du \rangle,
\]
with $\beta_c$ as above, and consider the equation
\begin{equation}\label{E:appmu}
\begin{cases}
	f_t = \mathcal{L}^{\eps,\delta} f + |Df|^2 & \;\;\text{in $Q$},\\
	f(P) = g(P) & \;\; \text{on $\Gamma$}.
\end{cases}
\end{equation}
We estimate the solutions of the approximations and then first let $\eps\to0$ and then $\delta\to0$.

\subsection{Lipschitz estimate in time}
We start with Lipschitz regularity in $t$ and then prove the regularity in $x$.

\begin{theorem}\label{T:lip_t0}
Suppose $g\in C^2(\overline{Q})$, $g\geq c > 0$ and $f = f^{\eps,\delta}$ is a smooth solution of \eqref{E:appmu}.
Then there exists $K_1>0$ depending only on $\|D^2g\|_\infty$, $\|Dg\|_\infty$, $\|g\|_\infty$ and
$\|g_t\|_{\infty}$ such that
\[
	|f(x,t) - g(x,0)| \leq K_1 t
\]
in $Q$. If $g$ is only continuous in $x$ and bounded in $t$ then the modulus of continuity of $f$
on $\Omega\times[0,t_*]$ (for small $t_*$) can be estimated in terms of $\|g\|_\infty$ and the modulus of
continuity of $g_0\bydef g(\cdot,0)$ in $x$.
\end{theorem}
\begin{proof}
Step 1. Assume $g \in C^2$ and $\lambda > 0$. If we define
\[
	v^+(x,t) = g_0(x) + \lambda (e^{\lambda t} - 1) = g(x,0) + \lambda(e^{\lambda t} - 1),
\]
then $v^+$ is a supersolution. Indeed,
\[\begin{aligned}
	v^+_t - \mathcal{L}^{\eps,\delta} & v^+ - |Dv^+|^2 \\
		&= \lambda^2 e^{\lambda t} - \eps\lap g_0 - k\frac{g_0 + \lambda (e^{\lambda t} - 1)}
			{|Dg_0|^2 + \delta^2} \langle D^2g_0\cdot Dg_0, Dg_0 \rangle - |Dg_0|^2 \\
	& \geq e^{\lambda t} \left( \lambda^2 - \|D^2g_0\|_{\infty} \lambda \right)
		- \left( (\eps +k\|g_0\|_{\infty})\|D^2g_0\|_{\infty} + \|Dg_0\|^2_\infty \right)\\
	&\geq 0
\end{aligned}\]
if we choose $\lambda >0$ large enough. In fact, if $\|D^2g_0\|_\infty \leq 1$ we can choose $\lambda > 1 +
\sqrt{(1+k\|g\|_\infty) + \|D^2g_0\|_\infty}$ and if $\|D^2g_0\| >1$ we can take $\lambda >
\max(2\|D^2g_0\|,1+k\|g_0\|_\infty + \|Dg_0\|^2_\infty)$.
Also, for $x\in\del \Omega$ and $t> 0$,
\[
	v^+(x,t) = g_0(x) + \lambda (e^{\lambda t} - 1) > g_0(x) + t \|g_t\|_\infty \geq g(x,t),
\]
if $\lambda$ is large, for example $\lambda > \sqrt{\|g_t\|_{\infty}}$. Clearly, $v^+(x,0) =
g(x,0)$, therefore, by the classical comparison (see for example \cite{LSU}),
\[
	f(x,t) \leq v^+(x,t) = g_0(x) + \lambda(e^{\lambda t} - 1),
\]
for every $(x,t) \in Q$. Similarly, with
\[
	v^- = g_0(x) - \lambda(e^{\lambda t} - 1)
\]
we obtain the symmetric inequality, and hence
\begin{equation}\label{E:f_lip_est}
	|f(x,t) - g_0(x)| \leq K_0 t
\end{equation}
for $t\in[0,T]$ where $K_0$ is a constant depending only on the stated norms of $g$ (as it is, it also depends
on $T$, but we can take this inequality on a bounded interval $[0,\theta]$ and then iterate it).

Step 2. Assume now $g_0$ is only continuous and let $\omega_0$ be its modulus of continuity. Let us fix a point
$x_0 \in \Omega$ and $0 < \rho < \min(\dist(x_0,\del \Omega),2\sqrt{\|g\|_\infty})$. Let us also define
\[
	g^{\pm}(x,t) = g_0(x_0) \pm \omega_0(\rho) \pm \frac{2\|g\|_\infty}{\rho^2}|x-x_0|^2.
\]
It is easy to see that $g^- \leq g \leq g^+$ on $\Gamma$ and thus, again from the comparison principle,
$f^- \leq f \leq f^+$, where $f^\pm$ is the solution of \eqref{E:appmu} with initial and boundary condition
$g^\pm$. Since $g^\pm$ are in $C^2(\Rd\times\R)$, we can use estimate \eqref{E:f_lip_est} to conclude that
\[
	|f^\pm(x_0,t) - g_0^\pm(x_0)| \leq K^+_0 t
\]
where $K^+_0$ depends on $\|g\|_\infty$ and $\rho$. Therefore,
\[\begin{aligned}
	|f&(x_0,t) - g_0(x_0)| \\
	& \leq |f(x_0,t) - f^\pm(x_0,t)| + |f^\pm(x_0,t) - g^\pm(x_0,0)| + | g^\pm(x_0,0) - g_0(x_0)| \\
	& \leq \oo{2} |f^+(x_0,t) - f^-(x_0,t)| + K^+_0 t + \omega_0(\rho) \\
	& \leq \oo{2} |f^+(x_0,t) - g^+(x_0,0)| + \oo{2} |f^-(x_0,t) - g^-(x_0,0)| \\
	& \qquad \qquad + \oo{2} |g^+(x_0,0) - g^-(x_0,0)| + K^+_0 t + \omega_0(\rho)\\
	& \leq 2 K^+_0 t + \frac{3}{2} \omega_0(\rho).
\end{aligned}\]
With this inequality it is straightforward to conclude the proof.
\end{proof}

The full Lipschitz estimate in time now follows easily.

\begin{theorem}\label{T:lip_t}
If $f$ is a solution if \eqref{E:appmu} in $Q$ and $g \in C^2(\overline{Q})$, then there exists $K_2 > 0$
depending only on $\|D^2g\|_\infty$, $\|Dg\|_\infty$, $\|g\|_\infty$ and $\|g_t\|_\infty$ such that
\[
	|f(x,t) - f(x,s)| \leq K_2 |t - s|
\]
for every $x\in \Omega$, $t,s\in (0,T)$. If $g$ is merely continuous, we can estimate the modulus of continuity of
$f$ on $Q$ in terms of $\|g\|_\infty$ and the modulus of continuity of $g$.
\end{theorem}
\begin{proof}
Taking $\tau >0$ and
\[
	\hat{f}(x,t) \bydef f(x, t+\tau),
\]
using the lemma it is immediate to get
\[
	|f(x,t) - \hat{f}(x,t)| \leq K_2 t
\]
in $\Omega\times [0,T-\tau]$. The case when $g$ is only continuous is done in a similar fashion to the previous proof.
\end{proof}

\subsection{H\"older continuity in space}

\begin{theorem}\label{T:holder_bdry}
Let $f$ be the strictly positive solution of \eqref{E:appmu} with $g\in C^2(Q)\cap \mathrm{Lip}(\overline{Q})$, $g \geq c >0$.
There exist
$\alpha\in(0,1)$, and constants $K_3 > 1$ and $\rho_3$, depending only on $\alpha$, $\|g\|_\infty$,
$\|Dg\|_\infty$, $\|g_t\|_\infty$ and $\alpha$, such that
for every $\eps$ and $\delta$ sufficiently small and for every $P_0 = (x_0,t_0) \in \Gamma$ and
$x\in \Omega$ with $|x-x_0| \leq \rho_3$ we have
\[
	|f(x,t_0) - g(x_0,t_0)| \leq K_3 |x-x_0|^\alpha.
\]
\end{theorem}
\begin{proof}
Step 1. Let us define
\[
	v^+(x,t) = g(x_0,t_0) + K_*|x-x_0|^\alpha + \lambda\left( t_0-t \right),
\]
where $K^* \geq 1$ and $\lambda >0$ are constants which we will choose in such a way as to make $v^+$
a supersolution lying above $f$ on the appropriate domain. Let us take $x \in \Omega\cap B_{\rho_3}(x_0)$
and $t \in (0,t_0)$. An easy computation, using the fact that $|x-x_0| \leq \rho_3 \leq 1$, yields
\[\begin{aligned}
	v^+_t &- \mathcal{L}^{\eps,\delta} v^+ - |Dv^+|^2 = -\lambda
		- K_*^2 \alpha^2 |x-x_0|^{2\alpha-2} \\
	&\qquad\quad + K_*\alpha |x-x_0|^{\alpha -2} \left( \frac{(1-\alpha) kv^+(x,t)}{1 + \left(
		\delta |x-x_0|^{1-\alpha}/K_*\alpha\right)^2} - \eps(d+\alpha-2) \right) \\
	&\geq -\lambda + \frac{K_*\alpha}{|x-x_0|^{2-\alpha}}\left( \frac{(1-\alpha) kc}{1 +
		\left(\delta/K_*\alpha\right)^2} - \eps(d+\alpha-2) - K_*\alpha|x-x_0|^\alpha \right).
\end{aligned}\]
Now note that if $\delta \leq \alpha < 1$ then $1 + \left(\delta/K_*\alpha\right)^2 \leq 2$; furthermore,
we have $\eps(d+\alpha-2) \leq (1-\alpha)kc/4$ for $\eps \leq \frac{(1-\alpha)kc}{4(d-1)}$, the
inequality being trivial if $d=1$. We can also choose $\rho_3 < 1$ such that
\[
	K_*\alpha|x-x_0|^\alpha \leq \frac{(1-\alpha)kc}{8}
\]
for $|x-x_0|\leq\rho_3$. Indeed, this is the case if we require that
\[
	\alpha K_* \rho_3^\alpha \leq (1-\alpha)kc/8.
\]
Hence, for $\delta$ and $\eps$ in the specified range and with this choice
for $\rho_3$, $K_*$ and $\alpha$, we have
\[
	v^+_t - \mathcal{L}^{\eps,\delta} v^+ - |Dv^+|^2 \geq -\lambda
		+ K_*\alpha\rho_3^{\alpha-2} \frac{(1-\alpha)kc}{8} > 0
\]
provided
\[
	K_*\alpha\rho_3^{\alpha-2} > \frac{8\lambda}{(1-\alpha)kc}.
\]

Step 2. We want to prove that $v^+ > f$ on $Q^* = (\Omega\cap B_{\rho_3}) \times (t_0 - t_*,t_0)$, where
we take $t_* \bydef \min\{1,t_0\}$. Let $P = (x,t) \in \del_p Q^*$. Let us first assume $P$ is on
the lateral boundary of $Q^*$. If $x\in \del \Omega$, then, since $f=g$ on $\del \Omega$ and $\rho_3 <1$,
\[\begin{aligned}
	f(P) &\leq f(P_0) + \|Dg\|_{\infty} |x-x_0| + \|g_t\|_{\infty} (t_0 -t) \\
	& \leq g(P_0) + K_*|x-x_0|^\alpha + \lambda(t_0 - t) = v^+(P),
\end{aligned}\]
provided $K_* \geq \|Dg\|_{\infty}\quad\text{and}\quad \lambda \geq \|g_t\|_{\infty}$.
If, on the other hand, $x\in \Omega\cap\del B_{\rho_3}(x_0)$, then, using comparison,
\[
	f(P) \leq \|g\|_{\infty} \leq f(P_0) + K_*\rho_3 + \lambda(t_0 - t)
		\leq v^+(P),
\]
provided $K_* \geq \|g\|_\infty / \rho_3$.

Step 3. We have to consider the case when $P$ is on the bottom
of the cylinder $Q^*$. Let us first assume $x\in \Omega\cap B_{\rho_3}(x_0)$ and $t = t_0 -1$. In this
case, again using comparison, we get
\[
	f(P) \leq \|g\|_{\infty} \leq f(P_0) + K_*|x-x_0|^\alpha + \lambda = v^+(P),
\]
as long as $\lambda \geq \|g\|_{\infty}$. Finally, when $t_0 <1$, and hence
$Q^* = (\Omega\cap B_{\rho_3}) \times (0,t_0)$, we have that $f = g$ on the bottom, therefore
\[\begin{aligned}
	f(P) &= f(x,0) = g(x,0) \leq g(x_0,t_0) + \|Dg\|_{\infty} |x-x_0| + \|g_t\|_{\infty} t_0 \\
	&\leq g(P_0) + K_* |x-x_0|^\alpha +\lambda t_0 = v^+(P),
\end{aligned}\]
provided, once again, $K_* \geq \|Dg\|_{\infty}$ and $\lambda \geq \|g_t\|_{\infty}$.

Step 4. To summarize, we have $v^+ \geq f$ on $\del_p Q^*$, and hence, by comparison, $v^+ \geq f$ in
$Q^*$, if
\[\begin{aligned}
	K_*\alpha\rho_3^\alpha &\leq \frac{(1-\alpha)kc}{8}, \\
	K_*\alpha\rho_3^{\alpha-2} &> \frac{8\lambda}{(1-\alpha)kc} \\
	\lambda &\geq \max\left\{\|g_t\|_\infty,\,\|g\|_\infty \right\}, \\
	K_* &\geq \max\left\{ \|Dg\|_\infty, \, \frac{\|g\|_\infty}{\rho_3} \right\}, \\
	\delta &\leq \alpha \quad \text{and} \quad \eps \leq \frac{(1-\alpha)kc}{4(d-1)}.
\end{aligned}\]
This can be achieved, for example, taking
\[ 
	\lambda = \max\left\{\|g_t\|_\infty,\,\|g\|_\infty \right\}, \;
	\rho_3 < \min\left\{ \frac{\|g\|_\infty}{\|Dg\|_\infty}, \frac{kc}{16}\right\},\;
	\alpha = \min\left\{\oo{2},\,\frac{kc}{16\|g\|_\infty}\sqrt{\rho_3}\right\}
	\;\text{and}\; K_* = \|g\|_\infty\rho_3^{-1}.
\]
Therefore, we have
\[
	f(x,t_0) - g(x_0,t_0) \leq v^+(x,t_0) - g(P_0) = K_*|x-x_0|^\alpha.
\]
Using the barrier $v^-\bydef g(P_0) - K_*|x-x_0|^\alpha + \lambda(t - t_0)$ we get the reverse
inequality,
\[
	f(x,t_0) - g(x_0,t_0) \geq - K_*|x-x_0|^\alpha.
\]
\end{proof}

We can extend the estimate to the interior of the domain. We will use the following notation for convenience.
For $z\in \Rd$, we define $\Omega_z = z + \Omega = \{x+z\mid x\in \Omega \}$ and for $r>0$, $\Omega_r = \{ x \in \Omega \mid \dist(x,
\del \Omega) \geq r \}$. Given $x,y\in\Rd$ we define the closed segment $[x,y] \bydef \{ \theta y + (1-\theta) x
\mid 0\leq \theta \leq 1 \}$. The semi-open and open segments $[x,y)$, $(x,y]$ and $(x,y)$ are defined
analogously.

\begin{theorem}\label{T:holder_x}
The conclusion of Theorem~\ref{T:holder_bdry} is valid in the interior of $Q$, that is, there exists
$K_4$, depending only on $\|g\|_{\infty}$, $\|Dg\|_{\infty}$ and $\|g_t\|_{\infty}$, such that for
$\eps$ and $\delta$ sufficiently small and for every $x,y\in \Omega$
\[
	|f(x,t) - f(y,t)| \leq K_4|x-y|^ \alpha.
\]
\end{theorem}
\begin{proof}
Step 1. Take a vector $z\in B_{\rho_3}(0)$ and let $V = \Omega \cap \Omega_z$. Let us define
$f_z(x) \bydef f(x-z)$. From Theorem~\ref{T:holder_bdry} we have that $|f(x)-f_z(x)| \leq K_3|z|^\alpha$ on
$\del V$ ($x\in \del V$ implies that $x\in \del \Omega$ or $x-z\in \del \Omega$). Hence, using the comparison
principle we have that $f_z(x) - K_3|z|^\alpha \leq f(x) \leq f_z(x) + K_3|z|^\alpha$ for $x\in V$. This means that whenever $x,y\in
\Omega_{x-y}\cap  \Omega$ or $x,y\in \Omega_{y-x} \cap \Omega$, with $|x-y| \leq \rho_3$, we have $|f(x) - f(y)| \leq K_3 |x-y|^\alpha$.
In particular, the same is true whenever $x,y\in \Omega_{|x-y|}$.

Step 2. When $|x-y| > \rho_3$, using the comparison principle we obtain the conclusion of the theorem taking
$K_4 = 2\|g\|_\infty/\rho_3^\alpha$. Let us therefore assume that $|x-y| \leq \rho_3$ and  $x-y\notin \Omega_{|x-y|}$.
Let us first further assume that $[x,y] \subset \Omega$. In this case we can take the two segments $[x,w]$ and
$[w,y]$, where $w = (x+y)/2$ is the midpoint of $[x,y]$, let $z = y-w$ and note that $w,y\in \Omega_z \cap \Omega$ and
$x,w\in \Omega_{-z}\cap \Omega$. Hence, from the first step of this proof, we have
\[\begin{aligned}
	|f(x) - f(y)| &\leq |f(x) - f(w)| + |f(w) - f(y)| \leq K_3(|x-w|^\alpha + |w-y|^\alpha) \\
	&\leq 2^{1-\alpha}K_3 |x-y|^\alpha.
\end{aligned}\]

Step 3. If the segment $[x,y]$ is not completely in $\Omega$, then we can certainly find $w_1, w_2 \in \del \Omega \cap [x,y]$
(not necessarily different) such that $[x,w_1) \in \Omega$ and $(w_2,y] \in \Omega$. In this case we can apply
Theorem~\ref{T:holder_bdry} directly to get $|f(x) - f(w_1)| \leq K_3 |x-w_1|^\alpha$ and $|f(w_2) - f(y)| \leq
K_3 |w_2-y|^\alpha$. Since $|f(w_1) - f(w_2)| \leq \|Dg\||w_1 - w_2| \leq K_3|w_1 - w_2|^\alpha$, we easily get
the result with $K_4 = 4^{1-\alpha} K_3$. This finishes the proof.
\end{proof}

\subsection{Lipschitz estimate in space}

Observe that, even though the function $(x,t) \mapsto K_* |x-x_0| + \lambda (t_0 -t)$ is a
viscosity supersolution of \eqref{E:mu} when $x_0 \in \del \Omega$---on the backward cylinder, and for
appropriate choices of the constants $K_*$ and $\lambda$---it is the $\eps\lap u$ term that prevents
this function from being a viscosity supersolution of \eqref{E:appmu}, as it gives rise to a bad
term with the wrong sign. It is therefore necessary to let $\eps \to 0$ to obtain the Lipschitz
estimate. We start by obtaining a Lipschitz estimate on the boundary.

\begin{theorem}\label{T:lip_bdry}
Let $g \in \mathrm{Lip}(\overline{Q})$, $g \geq c > 0$, and suppose $f$ is a (strictly positive) solution of
\begin{equation}\label{E:app0mu}
\begin{cases}
	f_t = \mathcal{L}^{0,\delta} f + |Df|^2  & \text{in $Q$}, \\
	f(P) = g(P) & \text{on $\Gamma$}.
\end{cases}
\end{equation}
There exist constants $K_5$ and $\rho_5$, depending only on $\|g\|_{\infty}$, $\|Dg\|_{\infty}$ and
$\|g_t\|_{\infty}$ (independent of $\delta\in(0,1)$), such that for every $P_0 = (x_0,t_0) \in \del
\Omega\times (0,T)$ and $x\in \Omega\cap B_{\rho_5}(x_0)$ we have
\[
	|f(x,t_0) - g(x_0,t_0)| \leq K_5 |x-x_0|.
\]
Moreover, if $g$ is only continuous, then the modulus of continuity of $f$ can be estimated in terms of
$\|g\|_{\infty}$ and the modulus of continuity of $g$.
\end{theorem}

\begin{proof}
Step 1. Let $K_*$, $L_*$, $\theta$ and $\lambda$ be positive constants and define
\[
	v^+(x,t) = g(P_0) + L_*|x-x_0| - K_*|x-x_0|^2 + \lambda(t_0 - t) + \theta.
\]
We will check $v^+$ is a viscosity strict supersolution on $Q^*$ ($Q^*$ defined as in the proof of
Theorem~\ref{T:holder_bdry}) above $g$ for an appropriate choice of constants $L_*$, $K_*$ and $\lambda$.
Another easy computation gives
\[\begin{aligned}
	v^+_t - &\mathcal{L}^{0,\delta} v^+ - |Dv^+|^2 \\
	&\geq -\lambda + \frac{2kcK_*}{1 + \left( \frac{\delta}{L_* - 2K_*|x-x_0|}\right)^2}
		- |L_* - 2K_*|x-x_0||^2 \\
	&\geq -\lambda + kcK_* - (L_*-1)^2 > 0,
\end{aligned}\]
provided $|x-x_0| \leq 1/2K_*$, $L_* \geq 2$ and $ K_* > ((L_*-1)^2 + \lambda)/(kc)$.

Step 2. We now need to choose the constants so that $v^+ > f$ on the parabolic boundary of $Q^*$. Take $P = (x,t)$
be a point in $\Gamma^* = \del_p Q^*$. If $x\in\del \Omega$, as before
\[\begin{aligned}
	f(P) &= g(P) \leq g(P_0) + \|Dg\|_{\infty} |x-x_0| + \|g_t\|_{\infty}(t_0-t) \\
	&\leq g(P_0) + (L_*-1)|x-x_0| + \lambda(t_0-t) < v^+(P),
\end{aligned}\]
provided $L_* \geq \|Dg\|_{\infty} + 1$ and $|x-x_0| \leq 1/2K_*$. If $x\in \Omega\cap \del B_{\rho_5}(x_0)$, then
\[
	f(P) \leq \|g\|_{\infty} \leq g(P_0) + (L_* -1)|x-x_0| + \lambda(t_0-t) \leq v^+(P),
\]
provided $L_* \geq \|g\|_{\infty} + 1$ and yet again $|x-x_0| \leq 1/2K_*$.

Step 3. When $P = (x,t)$ is on the bottom of the cylinder $Q^*$, as before we consider two cases. When $t_0 \geq 1$,
$t = t_0-1$ and
\[
	f(P) \leq \|g\|_{\infty} \leq g(P_0) + \lambda < v^+(P)
\]
as long as $\lambda \geq \|g_t\|_{\infty}$, $L_* \geq 1$ and $|x-x_0| \leq 1/2K_*$. On the other hand, if
$t_0 <1$, $t=0$ and hence
\[
	f(p) = g(P) \leq \|g\|_{\infty} \leq g(P_0) + \lambda < v^+(P)
\]
under the exact same conditions as for the previous formula.

Step 4. Hence we have $f \leq v^+$ on $\Gamma^*$, and thus by comparison on $Q^*$, as long as we take
\[\begin{aligned}
	\lambda &\geq \|g_t\|_{\infty}, \\
	L_* &\geq \max\left\{ 2, \, \|Dg\|_{\infty}+1, \, \|g\|_{\infty}+1 \right\}, \\
	K_* &> \frac{(L_* - 1)^2 + \lambda}{kc}, \\
	\rho_5 &\leq \oo{2K_*}.
\end{aligned}\]
Therefore, using once more the comparison principle, we have that for $x\in \Omega\cap B_{\rho_5}(x_0)$,
\[
	f(x,t_0) \leq v^+(x,t_0) \leq g(P_0) + L_*|x-x_0| + \theta.
\]
Since $\theta$ is arbitrary, we have
\[
	f(x,t_0) \leq g(P_0) + L_*|x-x_0|.
\]
Using instead the barriers
\[
	v^-(x,t) = g(P_0) - L_*|x-x_0| + K_*|x-x_0|^2 + \lambda(t-t_0) - \theta
\]
we obtain the reverse inequality, and as a consequence the Lipschitz estimate.

Step 5. Let us finally merely assume that $g$ is continuous and let $\omega_g(\sigma)$ be a modulus of continuity
at $P_0$. More specifically, let $\omega_g$ be a continuous, decreasing function in $\sigma$ such that
$|g(P) - g(P_0)| \leq \omega_g(\sigma)$ whenever $\max\{{|x-x_0|},\, |t-t_0|\} \leq \sigma$. Let $\sigma \in (0,t_0)$
and define the smooth functions
\[
	g^{\pm}(x,t) \bydef g(x_0,0) \pm \omega_g(\sigma) \pm \frac{4\|g\|_{\infty}}{\sigma^2} |x-x_0|^2
		\pm \frac{2\|g\|_{\infty}}{\sigma} |t - t_0|.
\]
If $\max\{|x-x_0|,\, |t-t_0|\} \leq \sigma$, then
\[
	g^{-} (P) \leq g(P_0) - \omega_g(\sigma) \leq g(P) \leq g(P_0) + \omega_g(\sigma) \leq g^+(P),
\]
and if $\max\{|x-x_0|,\, |t-t_0|\} \geq \sigma$ then
\[
	g^-(P) \leq -\|g\|_{\infty} \leq g(P) \leq \|g\|_{\infty} \leq g^+(P).
\]
Therefore, if $f^\pm$ are the solutions of \eqref{E:app0mu} with initial data $g^\pm$, by comparison
$f^- \leq f \leq f^+$ on $Q$. Since $f^\pm$ are smooth we can apply the first part of the theorem to deduce that
\[
	|f^\pm(x,t_0) - g^\pm(P_0)| \leq K_5^+ |x-x_0|,
\]
where $K_5^+$ depends on $\|g\|_{\infty}$ and $\sigma$. From these inequalities we get
\[\begin{aligned}
	|f(x,t_0) &- g(P_0)| \\
	&\leq |f(x,t_0) - f^\pm(x,t_0)| + |f^\pm(x,t_0) - g^\pm(P_0)| + |g^\pm(P_0) - g(P_0)| \\
	&\leq \oo{2}|f^+(x,t_0) - f^-(x,t_0)| + K_5^+|x-x_0| + \omega_g(\sigma) \\
	&\leq \oo{2}|f^+(x,t_0) - g^+(P_0)| + \oo{2}|f^-(x,t_0) - g^-(P_0)| \\
		&\qquad \oo{2}|g^+(P_0) - g^-(P_0)| + K_5^+|x-x_0| + \omega_g(\sigma) \\
	&\leq 2K_5^+|x-x_0| + \frac{3}{2}\omega_g(\sigma).
\end{aligned}\]
This finishes the proof.
\end{proof}

It is now straightforward to obtain the interior Lipschitz estimate.
\begin{theorem}\label{T:lip_x}
Let $g$ and $f$ be as in Theorem~\ref{T:lip_bdry}. For every $x,y\in \Omega$ and $t\in(0,T)$
\[
	|f(x,t) - f(y,t)| \leq K_5|x-y|,
\]
where $K_5$ is the constant given in that theorem.
If $g$ is only continuous, then the modulus of continuity of $x\mapsto f(x,t)$ can be estimated in terms
of $\|g\|_\infty$ and the modulus of continuity of $g$ in $x$.
\end{theorem}
\begin{proof}
Step 1. The proof is very similar to the proof of Theorem~\ref{T:holder_x}, but in this case it is easy to get
the optimal Lipschitz constant. Take $z\in\Rd$ such that $|z| \leq \rho_5$. Define $V = \Omega \cap \Omega_z$ and let
$f_z(x) \bydef f(x-z)$. From previous theorem we know that $|f(x)-f_z(x)| \leq K_5 |z|$ on $\del V$. Using
comparison, we have that $f_z(x) - K_5|z| \leq f(x) \leq f_z(x) + K_5|z|$ in $V$. Therefore,
$|f(x)-f(y)| \leq K_5|x-y|$ if $x,y\in \Omega_{x-y}$, and in particular the same is true if $x,y \in \Omega_{|x-y|}$.
$\Omega_r = \{ x\in \Omega \mid \dist(x,\del \Omega) >r\}$.

Step 2. Suppose now $x-y\notin \Omega_{|x-y|}$ and let us first assume that the whole segment $[x,y] = \{ z\in \Rd \mid z
= \theta y + (1-\theta) x, \; 0\leq\theta\leq 1 \}$ is in $\Omega$. Let us assume without loss of generality that
$\rho = \dist(x,\del \Omega) \leq \dist(y, \del \Omega)$. We can find points $x_i$, $0\leq i \leq n$ such that $x = x_0$,
$x_n = y$, $x_i \in [x_{i-1},x_{i+1}]$ ($1\leq i \leq n-1$), and $\rho_i = |x_i - x_{i-1}| \leq \rho$
($1\leq i \leq n$). Noting that $x_i,x_{i-1}\in \Omega_{x_i-x_{i-1}}$ we can use the previous step to conclude that
$|f(x_i) - f(x_i)| \leq K_5|x_i - x_{i-1}|$, and hence
\[
	|f(x) - f(y)| \leq \sum_{i=1}^n |f(x_i) - f(x_{i-1})| \leq K_5 |x-y|.
\]

Step 3. If, on the other hand $[x,y] \notin \Omega$, then we can find points $x_1, x_2 \in \del \Omega \cap [x,y]$ such that
$[x,x_1) \subset \Omega$ and $[x_2,y]\setminus\{x_2\} \subset \Omega$. We can further choose $w_1\in [x,x_1]$, with
$|w_1 - x_1| \leq \rho_5$ and $w_2 \in [x_2,y]$, with $|w_2 - x_2| \leq \rho_5$. Then we apply step 2 above to
obtain $|f(x) - f(w_1)| \leq K_5 |x-w_1|$, $|f(w_2) - f(y)| \leq |w_2-y|$, while from the previous theorem,
$|f(w_i) - f(x_i)| \leq K_5 |w_i - x_i|$. Putting all these inequalities together gives the Lipschitz estimate
for this last case.

The proof of the statement with the modulus of continuity follows as in the proof of
Theorem~\ref{T:holder_x}.
\end{proof}

We finally prove Theorem~\ref{T:exist_pos}. Existence can be proved piecing out the results in theorems
\ref{T:lip_t}, \ref{T:holder_x} and \ref{T:lip_x}, and using the standard compactness arguments. Uniqueness
follows directly from the strong comparison result, Theorem~\ref{T:comparison_pos_strong}.

\begin{proof}[Proof of Theorem~\ref{T:exist_pos}]
Step 1. Assume first $g\in C^2(Q)\cap \mathrm{Lip}(\overline{Q})$. The comparison principle and theorems~\ref{T:lip_t}
and~\ref{T:holder_x} imply that the family of functions $\{f^{\eps,\delta}\}$ is uniformly bounded and
equicontinuous, therefore, for some sequence $\eps_k \to 0$, $f^{\eps_k,\delta} \to f^\delta$, which, by a
standard argument of viscosity solutions (see the proof of Theorem~\ref{T:exist_maximal} above for a similar argument), is a solution of \eqref{E:app0mu}.

Using now theorems~\ref{T:lip_t} and~\ref{T:lip_x} (note that the estimate in Theorem~\ref{T:lip_t} is
independent of $\eps$ and $\delta$), we can use the same compactness argument to find a sequence $f^{\delta_k}
\to f$. The usual arguments for viscosity solutions work here to show that viscosity subsolution of \eqref{E:mu}
and satisfies condition $(i)$ of Definition~\ref{D:visc_supsol} (we make the details explicit in the similar
argument for Theorem~\ref{T:exist_maximal} below). The second condition is trivial because our
functions are strictly positive.

Step 2. Let us suppose that $v$ is a classical moving free-boundary subsolution of \eqref{E:mu} which is strictly
separated from $f$ at $t_1$ and on the portion of the lateral boundary of $Q$ with times between $t_1$ and $t_2$,
$\Gamma_{t_1,t_2}$ (see Definition~\ref{D:visc_supsol}). Suppose that $v$ crosses $f$ for the first time at an
interior point of $Q_{t_1,t_2} = Q \cap \Rd\times[t_1,t_2)$, $P_* = (x_*,t_*)$. Since $f(P_*) > 0$, we get a
contradiction from the comparison result, Theorem~\ref{T:comparison_pos}, applied to a sufficiently small
cylinder around $P_*$, choosing $c$ appropriately.

Step 3. Since the estimates for the Lipschitz constant (or modulus of continuity) of the approximation $f^\delta$ is
independent of $\delta$, and the convergence is locally uniform, the statements of the theorem concerning
Lipschitz continuity (respectively modulus of continuity) readily follow.
\end{proof}

\section{The Dirichlet problem in a bounded domain. Large time behaviour}
\label{sec.asymp.bdd}

If $\Omega$ is bounded but not a ball, or the data are not radially symmetric, then the exact asymptotic behaviour is not the same as in the standard (PME). We treat the general situation here, where
 $\Omega\subset\Rd$ is bounded and
\begin{equation}\label{E:g_Diri}
	g_0(x,t) = \begin{cases}
		0 & \text{if $x\in\del\Omega$}, \\
		u_0(x) & \text{if $t = 0$},
	\end{cases}
\end{equation}
and $u_0$ is a Lipschitz continuous function, positive in $\Omega$ and vanishing on $\del \Omega$. The Dirichlet problem in $Q=\Omega\times [0,T]$ is
\begin{equation}\label{E:DP}
	\begin{cases}
		u &\;\text{solves \eqref{E:mu} in $Q$ and} \\
		u = g_0 &\;\text{on $\Gamma = \del Q \times[0,T] \cup Q\times\{t=0\}$.}
	\end{cases}
\end{equation}
Here is a first result on asymptotic behaviour.

\begin{prop}\label{T:asymptotic_1}
Let $u$ be the maximal solution of \eqref{E:DP}. Then $u$ decays like $O(t^{-1})$. More precisely, we have
\[
	[F_1(x)]^{m-1} \leq t\, u(x,t) \leq [F_2(x)]^{m-1}
\]
where $F_i(x)$ is the profile given by \eqref{E:sep_F1} for two different radii $R_i$, $i=1,2$.
\end{prop}
\begin{proof}
Take $x_0\in\Omega$ and let $B_1=B_{R_1}(x_0)$ and $B_2=B_{R_2}(x_0)$ be such that $B_1 \Subset \Omega \subset B_2$. Let
$u_1$ and $u_2$ be the solutions given by \eqref{E:stythos_u} with $t_0 = t_i$ and $R = R_i$, $i=1,2$, respectively, and
choosing $t_i<0$ such that $u_1 < u_0 < u_2$ at $t=0$. The theorem follows from the comparison result.
\end{proof}
\begin{rmk}
If $\Omega$ is a ball, then we can successively take the balls $B_1$ and $B_2$ sandwiching $\Omega$ and obtain the exact
asymptotic behaviour $u(x,t) \sim t^{-1}[F(x)]^{m-1}$, where $F$ is the profile for the ball $\Omega$ given in
\eqref{E:sep_F1}.
\end{rmk}

In order to get a more precise asymptotic behaviour we need to
 solve the  nonlinear elliptic problem
\begin{equation}\label{E:eevp}
	-\lapi F^m = \lambda F
\end{equation}
in $\Omega$. We will obtain the result backwards, that is, we will prove that there is an asymptotic profile and then
show that it solves the elliptic problem \eqref{E:eevp} with null boundary conditions.

In order to prove that in fact there is a profile, we need the following lemma which is adapted from a similar result
for the (PME) of B\'enilan-Crandall \cite{BC} (see \cite[Lemma~8.1]{V} for a simple proof for the (PME)).
\begin{lemma}
The maximal viscosity solution of the Dirichlet problem \eqref{E:DP}, satisfies
\begin{equation}\label{E:ut}
	u_t \geq -\frac{u}{t}.
\end{equation}
\end{lemma}
\begin{proof}
Let us for convenience define for $\lambda \geq 1$ the operator $\mathcal{T}_\lambda:C(Q) \mapsto C(Q_\lambda)$,
$Q_\lambda = \Omega\times[0,\lambda T]$,
\[
	\mathcal{T}_\lambda[u](x,t) \bydef \lambda u(x,\lambda t).
\]
Consider the approximation $u^n$ from the proof of Theorem~\ref{T:exist_maximal}, Section~\ref{S:maximal_solutions}, and
define $u^n_{\lambda} = \mathcal{T}_\lambda[u^n]$. It is straightforward to check that $u^n_{\lambda}$ solves
\eqref{E:mu} with data $g^n_{\lambda} = \lambda g^n(x,\lambda t)$. Therefore, for $\lambda >1$, using comparison, we
see that $u^n_\lambda \geq u^n$ in $Q$. Hence, from Theorem~\ref{T:independent_approx} $u_\lambda = \lim_{n}
u^n_\lambda$ is the maximal solution of \eqref{E:DP}
with data $g_\lambda = \lambda g(x, \lambda t)$ and satisfies $u_\lambda \geq u$ in $Q$. In fact, we must have
$u_\lambda = \mathcal{T}_\lambda[u]$, since $u^n\to u$ implies $u^n_\lambda = \mathcal{T}_\lambda[u^n] \to
\mathcal{T}_\lambda[u]$.

Since $u_\lambda \geq u$ and these functions are Lipschitz (from Theorem~\ref{T:exist_pos}), we can compute
\[
	0 \leq \left.\frac{du_\lambda(x,t)}{d\lambda} \right|_{\lambda = 1^+} = u(x,t) + tu_t(x,t),
\]
which immediately gives \eqref{E:ut}
\end{proof}
We are now in position to establish the following:
\begin{proof}[Proof of Theorem~\ref{T:large_time_Dirichlet}]
The above estimates for $u$ and $u_t$ suggest the following rescaling:
\begin{equation}\label{E:v_rescale}
	v(x,\tau) = \mathcal{V}(u) (x,\tau) \bydef
		\left[ \alpha e^{(m-1)\tau} u\left(x, e^{(m-1)\tau}\right) \right]^{\oo{m-1}},
\end{equation}
where $\alpha = (m-1)^2/m$. The bounds $t^{-1} u(x,t) \leq C$, from
Theorem~\ref{T:asymptotic_1}, and \eqref{E:ut} in terms of $v$ are
\[
	v(x,\tau) \leq C \;\;\;\text{and}\;\;\; v_\tau \geq 0.
\]
It is not hard to check that in fact, at least formally, $v$ solves
\begin{equation}\label{E:v_eq}
 v_t = \lapi v^m + v.
\end{equation}
The estimates for $v$ now imply that there exists $G = G_\Omega$ such that
\[
	\lim_{\tau\to\infty} v(x,\tau) = G(x).
\]
The function $F_\Omega$ from the statement of the theorem is just $F_\Omega = \alpha^{-\oo{m-1}} G =
\alpha^{-\oo{m-1}} G_\Omega$.
\end{proof}
\begin{theorem}\label{T:exist_eevp}
The function $G_\Omega(x)$ from the previous proof is a positive viscosity solution of the eigenvalue problem
\eqref{E:eevp} in $\Omega$ with $\lambda = 1$ and null boundary condition.
\end{theorem}
\begin{proof}
Step 1. Let $\phi$ be a smooth function touching $G$ from above at a point $x_0\in\Omega$ and assume as usual that $G-\phi$
has an absolute maximum at $x_0$. Define the functions $\eta, \eta_\eps\in C^1(\R)$ by
\[ \eta(s) \bydef \begin{cases}
	s^2 & \text{if $|s|<1$,} \\
	2|s|-1 & \text{if $|s|\geq 1$},
	\end{cases} \quad\;\;\text{and}\;\;\quad
	\eta_\eps(s) \bydef \eps\eta\left(s-\oo{\eps}\right).
\]
Define also $\phi_\eps(x,\tau) \bydef \phi(x) + \eta_\eps(\tau)$ and
\[
	\psi_\eps(x,t) \bydef \mathcal{V}^{-1}(\phi_\eps)(x,t) =
		\oo{\alpha t} \phi_\eps\left(x,\log\left[t^\oo{m-1}\right]\right)^{m-1}.
\]
Now observe that $v(x,\tau) - \phi_\eps(x,\tau) \leq G(x) - \phi(x) \leq 0$, therefore, there exists $c_\eps \geq 0$
such that $v - (\phi_\eps - c_\eps)$ has a maximum $0$ at a point $(x_\eps,\tau_\eps)$. Since $v_\tau \geq 0$, it
must be that $\tau_\eps \geq 1/\eps$, as for $\tau < 1/\eps$, $\phi_{\eps,\tau}(x,\tau) < 0$. Noting that the
convergence of $v$ as $\tau\to\infty$ is monotone, the uniform Lipschitz estimates imply that $G$ is Lipschitz and
the convergence is uniform. Therefore, we must have $(x_\eps,\tau_\eps) \to (x_0,\infty)$ and $c_\eps \to 0$ as
$\eps\to 0$.

Step  2. The transformation $\mathcal{V}$ is monotone in the sense that, with $\tau = \log(t^{\oo{m-1}})$,
$\mathcal{V}(u_1)(x,\tau) < \mathcal{V}(u_2)(x,\tau)$ whenever $u_1(x,t) < u_2(x,t)$. Hence, the function $\psi_\eps$
must touch $u$ from above at the point $(x_\eps,t_\eps)$, $t_\eps = \exp\{(m-1)\tau_\eps\}$ and consequently, at this
point,
\[\begin{aligned}
	& \psi_{\eps,t} \leq k\psi_\eps\lapi\psi_\eps + |D\psi_\eps|^2 && \quad\text{if $D\psi_\eps\neq 0$}, \\
	& \psi_{\eps,t} \leq k\psi_\eps\Lambda(D^2\psi_\eps) && \quad\text{if $D\psi_\eps = 0$}.
\end{aligned}\]
This in terms of $\phi_\eps$ translates into
\[\begin{aligned}
	& \phi_{\eps,\tau} \leq \lapi \phi_\eps^m + \phi_\eps && \quad\text{if $D\phi_\eps \neq 0$}, \\
	& \phi_{\eps,\tau} - \phi_\eps \leq \Lambda\left(D^2 \phi_\eps^m\right) && \quad\text{if $D\phi_\eps = 0$}
\end{aligned}\]
at the point $(x_\eps,\tau_\eps)$. Here we used the fact that when $Dw = 0$, $D^2(w^l) = l w^{l-1} D^2 w$.
Now let $\eps\to 0$. Observing that $|\phi_{\eps,\tau}| = |\eps\eta'| \leq 2\eps$ and $D\phi_\eps = D\phi$, we get
\[\begin{aligned}
	& 0 \leq \lapi \phi^m + \phi && \quad\text{if $D\phi\neq 0$,} \\
	& 0 \leq \Lambda\left( D^2\phi^m \right) + \phi && \quad\text{if $D\phi = 0$.}
\end{aligned}\]
This shows $G$ is a viscosity subsolution of \eqref{E:eevp} according to the definition of \cite{CGG}.

Step 3. To prove $G$ is a viscosity supersolution take now a smooth function $\phi$ touching $G$ from below at a point
$x_0\in\Omega$. Let
\[
	\phi_\eps(x,\tau) \bydef \phi(x) - \eps - \oo{\eps\tau^2} + \frac{2\sigma_\eps}{\eps\tau},
\]
where $\sigma_\eps$ is chosen so that $\phi_\eps$ touches $v$ from below at some point $(x_\eps,\tau_\eps)$. Let us
check that such a choice is indeed possible. First observe that for $\sigma_\eps = \eps$ we get
\[
	\phi_\eps\left(x_0,\oo{\eps}\right) = \phi(x_0) = G(x_0) \geq v\left(x_0,\oo{\eps}\right),
\]
that is, $\phi_\eps$ crosses or touches $v$. On the other hand, letting $\tau_\eps^*$ be such that $\tau > \tau_\eps^*$
implies $v(x,\tau) > G(x) - \eps$ (recall that the convergence as $\tau\to\infty$ is uniform), if $\sigma_\eps <
-(\tau_\eps^*)^2\eps L_v/2$, where $L_v$ is the Lipschitz constant for $v$, then $\phi_\eps < v$. Indeed, for
$\tau\geq\tau_\eps^*$ we have $\phi_\eps(x,\tau) < G(x) - \eps < v(x,\tau)$ and for $\tau < \tau_\eps^*$ we compute
$\phi_{\eps,\tau}(x,\tau) = \frac{2}{\eps\tau^3} - \frac{2\sigma_\eps}{\eps\tau^2} > L_v$, which means we must have
$v(x,\tau) > \phi(x,\tau)$. Therefore, there must exist $\sigma_\eps \in \left( -(\tau_\eps^*)^2\eps L_v/2, \eps\right)$
such that $\phi_\eps$ touches $v$ from below at a point $(x_\eps,\tau_\eps)$. Note also that the above implies
$\tau_\eps > 1/\eps$, hence we have $\lim_{\eps\to 0} (x_\eps,\tau_\eps) = (x_0,+\infty)$. The rest of the argument goes
as in step 2 above.
\end{proof}

\section{The Cauchy problem}\label{S:Cauchy}
In this section we consider briefly the Cauchy problem
\begin{equation}\label{E:mu_cauchy}
	\begin{cases}
	u_t = k u \lapi u + |Du|^2, & \text{in $\Rd\times[0,T]$}, \\
	u(x,0) = u_0(x), & \text{for $x\in\Rd$}.
	\end{cases}
\end{equation}
When $u_0:\Rd\to\R$ is uniformly continuous, positive and has compact support we can readily apply the foregoing theory in a sufficiently large cylinder to obtain existence and uniqueness of a maximal solution, which necessarily has compact support as well. On the other hand, for positive data $u_0\geq 0 $, we can obtain the solution by considering first strictly positive data and approximating the problem posing it in $\Omega_r = B_{2r}(0) \times [0,T]$,
\begin{equation}\label{E:appmu_cauchy}
	\begin{cases}
	u_t = k \beta_c(u) \lapi u + |Du|^2, & \text{in $\Omega_r$}, \\
	u(x,t) = u_0^r(x), & \text{if $(x,t)\in\del_p \Omega_r$}.
	\end{cases}
\end{equation}
In principle we could take $u_0^r(x)=u_0(x)$, but for technical reasons we will have it depend on $r$. In any case,
we choose it so that $\{u_0^r\}$ is uniformly equicontinuous on compact sets, $u_0^r \geq c > 0$ and its modulus of
continuity is bounded by the modulus of $u_0$. For each $r>0$, let
\[
	u_0^r(x) = \begin{cases}
		u_0(x) & \text{if $|x|\leq r$,} \\
		M & \text{if $|x| = 2r$,} \\
		\max\left\{ u_0(x), M + \lambda\left[ u_0\left( r\frac{x}{|x|}\right) - M \right] \right\}
			& \text{if $|x| = 2r -\lambda r$}.
	\end{cases}
\]
\begin{theorem} Given $u_0:\Rd\to\R$ uniformly continuous, bounded and such that $u_0(x) \geq c >0$, there exists a
maximal bounded viscosity solution of \eqref{E:mu_cauchy}. The modulus of continuity of $u$ can be estimated in terms of the
modulus of continuity of $u_0$ and $\|u_0\|_\infty$.
\end{theorem}
\begin{proof}
Step 1. Let $u_r$ be the unique solution of \eqref{E:appmu_cauchy} given by Theorems~\ref{T:exist_pos}
and~\ref{T:comparison_pos_strong}. Fix $R>0$ and consider the family $\{u^r\}_{r>R}$. By the estimates of
Theorem~\ref{T:exist_pos}, this family is uniformly bounded and equicontinuous on $\Omega_R$. Therefore, using a diagonal
argument, we can extract a sequence which converges locally uniformly on $\Rd\times[0,T]$ to a continuous function $u$
whose modulus of continuity depends on the modulus of continuity of $u_0$ and $\|u_0\|_\infty$ (and is Lipschitz
continuous whenever $u_0$ has this property). Note also that $\inf_{\Rd} u = \liminf_r \inf_{\Rd} u_r \geq c > 0$.
The proof that $u$ is a viscosity solution goes as in the proof of Theorem~\ref{T:exist_maximal} above.

Step 2. To show that the solution obtained in Step 1 is maximal, let $v$ be another solution. We prove in the next lemma
that $v \leq M$. Since $u^r \geq M$ on the lateral boundary and $u_0^r \geq u_0 = v$ on $\{ t=0\}$, by comparison we
immediately obtain the result.
\end{proof}

In the proof of maximality above, we have used the following version of the technical Lemma~5.2 from \cite{BV}.
\begin{lemma}
If $u$ is a bounded viscosity solution of \eqref{E:mu_cauchy} such that $0 < c \leq u_0 \leq M$, then $u \leq M$
everywhere.
\end{lemma}
\begin{proof}
We consider the same function as in \cite[Lemma~5.2]{BV}
\[
	V(x,t) = N + b\frac{|x|^2}{2T-t} + \lambda t
\]
and easily check that $V$ is a classical strict supersolution if
\[
	1 > 2 (k+2) b \;\;\text{and}\;\; \lambda T (1-2kb) > 2kNb.
\]
These inequalities are satisfied if we select $b$ small enough and $\lambda$ not too small. In particular, this is true
for $b=\eps$, $N=M+\eps$ and $\lambda = \frac{8kM}{T} \eps = O(\eps)$ with $\eps$ small enough. This function is
strictly greater than $u_0$ at $t=0$ and, because we are assuming $u$ is bounded, it is also strictly above $u$ for
all large $|x|$ uniformly in $[0,T]$. If $V$ is not strictly larger than $u$ everywhere, there exists $P_1 = (x_1,t_1)
\in \Rd\times[0,T]$ where $V$ touches $u$ from above. At this point, from the definition of viscosity solution, we have
\[
	V_t(P_1) \leq kV(P_1)\lapi V(P_1) + |DV(P_1)|^2.
\]
This contradicts the fact that $V$ is a strict supersolution in a small parabolic neighbourhood of $P_1$, hence we
conclude that $u \leq V$ on $\Rd\times[0,T]$. Letting $\eps \to 0$ we conclude that $u \leq M$ everywhere.
\end{proof}

Using the same ideas as above, we can let $c\to 0$ to obtain maximal solution for nonnegative data.
\begin{theorem}
Given $u_0:\Rd\to\R$ uniformly continuous, bounded and such that $u_0(x) \geq 0$, there exists a
maximal viscosity solution of \eqref{E:mu_cauchy}. The modulus of continuity of $u$ can be estimated in terms of the
modulus of continuity of $u_0$ and $\|u_0\|_\infty$.
\end{theorem}

\subsection{Large time behaviour}

Considering data $u_0$ with compact support, using the Barenblatt functions and the traveling waves solutions, it is
interesting that the ideas of the classical (PME) can be used to obtain properties for the support of super- and
subsolutions and in particular to obtain the asymptotic behaviour as $t\to\infty$.

\begin{prop}\label{P:supp_sols}
We have the following properties:
\begin{itemize}
\item[(a)] The support of any viscosity supersolution is nondecreasing in time and penetrates the whole space
as $t\to\infty$.
\item[(b)] The support of any viscosity solution expands in a continuous way.
\end{itemize}
In particular, any solution whose initial condition has compact support, has support in a set expanding like $O(t^\oo{m+1})$ as $t\to\infty$.
\end{prop}
The proofs of \textit{(a)} and \textit{(b)} are identical to the proofs of \cite[Propositions~6.1 and 6.2]{BV} and the last statement follows from sandwiching the support of the solution between to Barenblatt functions. We can now prove our final theorem.

\begin{proof}[Proof of Theorem~\ref{thm.asbeh}]
Step 1. The first step is to sandwich $\rho(x,t_1)$ for some $t_1>0$ between
two 1-$d$ Barenblatt solutions  with 1-$d$ masses $0<M_1<M_2$ (note that $M_i=2\int_0^\infty \rho_i(r,0)dr$), and this has been explained above. Such bounds will preserved for all times $t\geq t_1$.

Step 2. Next, we copy from the study of asymptotic  behaviour of the 1-$d$ PME to define a family of rescaled solutions $\rho_\lambda$ for all $\lambda\geq 1$ as follows
\begin{equation}
\rho_\lambda(x,t) =\lambda^{\frac{1}{m+1}} \rho(\lambda^{\frac{1}{m+1}} x, \lambda t).
\end{equation}
It is easy to see that the $\rho_\lambda$ are maximal viscosity solutions of the Cauchy Problem (with rescaled initial data).
The bounds with Barenblatt solutions imply that on any compact time interval $[t_1,t_2]$ with $0<t_1<t_2<\infty$ the family
$\rho_\lambda$ is continuous, uniformly bounded, and supported in a uniform ball $B_{R_*}(0)$.

Step 3. Now we use Aleksandrov's principle, as explained for instance in \cite{CVW}, to show that
for a solution $\rho(x,t)$ with initial data $\rho_0(x)\geq 0$ supported in the ball $B_{R}(0)$ we have
for all $t\geq 0$ and all $r>R$
\[
	\inf_{|x|=r} \rho(x,t)= \max_{|x|=r+2R} \rho(x,t)
\]
If this is applied to the rescaled solutions, we get for all $|x|\geq R_\lambda = R\,\lambda^{-1/(m+1)}$
\[
	\inf_{|x|=r} \rho_\lambda(x,t)= \max_{|x|=r+2R_\lambda} \rho_\lambda(x,t)
\]

Step 4. We now fix $t=1$, $\lambda $ very large, so that $R_\lambda\leq \eps$ is very small, and define
\[
	\tilde\rho_1(r)=\inf_{|x|=r} \rho_\lambda(x,1), \quad
	\tilde\rho_2(r)=\max_{|x|=r} \rho_\lambda(x,1),
\]
We easily verify that $\tilde\rho_2(r),\, \tilde\rho_2(r)$ are nonnegative and radially symmetric functions,
both supported in the same ball $B_{R_*}(0)$, they are nonincreasing as functions of $r$ for $r\geq \eps$,
and we also have
\[
	\tilde\rho_2(r)\geq \tilde\rho_1(r)\geq \rho_1(r+\eps)
\]
for all $r\geq \eps$. It is then easy to verify that the 1-$d$ mass of $\rho_2(r) - \rho_1(r)$ is less than
$C\eps $.

Step 5. If $\tilde\rho_1(r,t)$ and $\tilde\rho_2(r,t)$ are the corresponding solutions of the 1-$d$ PME with
initial data at $t=1$ given by $\tilde\rho_1(r)$ and $\tilde\rho_2(r)$, respectively, we have for all $t\geq 1$
\[
	\tilde\rho_1(r,t)\leq \rho_\lambda(x,t)\leq \tilde\rho_2(r,t)
\]
In view of the convergence result for the 1-$d$ PME, cf.\ \cite[Theorem 18.1]{V}, the result follows.
\end{proof}

\appendix

\section{Radial solutions with separation of variables}

Let us look more closely at solutions of the form
\[
	\rho(x,t) = T(t)F(x)
\]
with $T$ and $F$ as in \eqref{E:sep_T} and \eqref{E:sep_F}. We further assume that $F$ is radial, $F(x) = F(r)$, with $r = |x|$.

\subsection{Radial solutions with \mbox{$\lambda >0$}}

Let $g(r) = \lambda^{- \frac{m}{m-1}} F^m(r)$. Then from \eqref{E:sep_F} we get
\[
	g{''} + g^p = 0,
\]
This equation can be integrated and yields
\[
	\frac{g(r)}{\sqrt{a}} \Fto\left(\oo{p+1},\oo{2};1+\oo{p+1},\frac{g^{p+1}(r)}{a}\right)
		= \pm kr + C,
\]
where $\Fto$ is the hypergeometric function and $C$ and $a$ are arbitrary constants, $a>0$, and
$k=\sqrt{\frac{2}{p+1}}$. Let us for convenience define $H_p(z) \bydef \frac{z}{\sqrt{a}}
\Fto\left(\oo{p+1},\oo{2},1+\oo{p+1},\frac{z^{p+1}}{a}\right)$, which satisfies $H_p'(z) =
\frac{1}{\sqrt{a - z^{p+1}}}$.
Note that $H_p$ is left-continuous at $z=a^{\oo{p+1}}$. In fact, we can compute
\[ H_p(0) = 0 \;\;\text{and}\;\; H_p(a^{\oo{p+1}}) = a^{\oo{p+1}-\oo{2}} \frac{\sqrt{\pi}\,
		\Gamma\left( 1 + \oo{p+1} \right)}{\Gamma\left(\oo{2} + \oo{p+1} \right)} \defby A_p.
\]
We have solved the equation for $g$ implicitly, $H_p(g(r)) = \pm kr + C$, where $k = \sqrt{2/(p+1)}$.
Since we want $g\geq0$, we need to define $H_p$ on the interval $I_p = [0,a^{\oo{p+1}}]$.
On this interval, $H'_p > 0$, therefore $H_p$ has an inverse, $G_p:J_p\to I_p$, where $J_p=[0,A_p]$,
and $G'_p > 0$. Then we have
\begin{equation}\label{E:g_sep_var}
	g(r) = G_p(\pm kr +C).
\end{equation}
Depending on whether we choose the plus or minus sign on this identity, we get two different types
of solution.

\subsubsection{Solutions defined on a ball}
Taking the minus sign in \eqref{E:g_sep_var} we get $g(r) = G_p(k(R-r))$, which is defined and smooth for
$r\in [0,R]$ with $kR = A_p$, that is
\[
	R = \frac{\sqrt{\pi}\,\Gamma\left( 1 + \oo{p+1} \right)}{\Gamma\left(\oo{2} + \oo{p+1} \right)}
		\sqrt{\frac{2}{p+1}} \,\,a^{\frac{m-1}{2(m+1)}}.
\]
The function $g$ is radially decreasing, vanishes at $r=R$ and has a maximum at $r=0$, $g(0)= a^{\oo{p+1}}
= M R^{\frac{2}{m-1}}$, where $M$ is a constant depending on $m$.

The solution $\rho(x,t) = T(t)F(x)$ is given by \eqref{E:sep_T} and
\begin{equation}\label{E:sep_F1}
	F(x) = \lambda^{\oo{m-1}} \left[ G_p \left(\sqrt{\frac{2m}{m+1}}(R -|x|) \right)
		\right]^{\oo{m}},
\end{equation}
that is,
\begin{equation}\label{E:stythos_rho}
	\rho(x,t) = \oo{[(m-1)(t-t_0)]^{\oo{m-1}}} \left[ G_p
		\left(\sqrt{\frac{2m}{m+1}}(R - |x|) \right) \right]^{\oo{m}}.
\end{equation}
For the transformed equation \eqref{E:mu} we get the solution
\begin{equation}\label{E:stythos_u}
	u(x,t) = \frac{m}{(m-1)^2(t-t_0)} \left[ G_p \left(\sqrt{\frac{2m}{m+1}}(R - |x|) \right)
		\right]^{\frac{m-1}{m}}.
\end{equation}
Note that $u$ and $\rho$ have infinite gradient at the boundary of the domain $\{|x| = R\}$.

\subsubsection{Solutions defined on an annulus}
If we take the plus sign in \eqref{E:g_sep_var} instead, we get the solutions defined on an annulus
$\{R_1 < r < R_2\}$, where $R_2 = A_p/k$. These solutions are
\begin{equation}\label{E:colon_rho}
	\rho(x,t) = \oo{[(m-1)(t-t_0)]^{\oo{m-1}}} \left[ G_p
		\left(\sqrt{\frac{2m}{m+1}}(|x|-R_1) \right) \right]^{\oo{m}},
\end{equation}
for \eqref{E:ipme}, and
\begin{equation}\label{E:colon_u}
	u(x,t) = \frac{m}{(m-1)^2(t-t_0)} \left[ G_p \left(\sqrt{\frac{2m}{m+1}}(|x|-R_1) \right)
		\right]^{\frac{m-1}{m}}
\end{equation}
for \eqref{E:mu}. These functions vanish on the inner boundary of the annulus, are radially increasing
and have maxima on the outer boundary equal to the maxima which the solutions of \eqref{E:stythos_rho} and
\eqref{E:stythos_u} attain at $x=0$. If $-A_p/k < R_1 \leq 0$, then the same formulas define solutions on
the annulus $\{0<r<R_2\}$, which are continuous but not differentiable at $x=0$.

\subsection{Radial solutions with \mbox{$\lambda <0$}}

For negative $\lambda$, we let $\ell = |\lambda|$ and still assuming that $F$ is radial, we let
$g(r) = \ell^{-\frac{m}{m-1}} F^m(r)$. Then $g$ solves
\[
	g{''} - g^p = 0.
\]
Integrating once as in the previous case, we now get
\[
	\frac{g{'}}{\sqrt{a + g^{p+1}}} = \pm\sqrt{\frac{2}{p+1}}.
\]
We now need to distinguish three cases: $a>0$, $a=0$ and $a<0$.

\subsubsection{Case $a>0$}
We compute $I_p(g(r)) = \pm\sqrt{\frac{2}{p+1}} r + C$, where $I_p(z) = \frac{z}{\sqrt{a}}
\sideset{_{2}}{_{1}}{\operatorname{F}} \left(\oo{p+1},\oo{2}; 1+\oo{p+1},-\frac{z^{p+1}}{a}\right)$.
This function is 1--to--1 from $[0,\infty)$ onto itself with $I_p(0) = 0$, $I_p{'}(0) =
\oo{\sqrt{a}}$, $I_p(z) \sim \frac{2}{1-p} z^{\frac{1-p}{2}}$ as $z\to +\infty$. Let $J_p$ denote the
inverse of $I_p$ on this interval. The function $J_p$ is also a bijection of $[0,+\infty)$ onto itself,
$J_p(0) = 0$, $J_p{'}(0) = \sqrt{a}$, $J_p(z)$ grows like $z^{\frac{2}{1-p}}$ as $z\to +\infty$ and
$J_p{'}(z)$ grows like $z^{\frac{1+p}{1-p}} = z^{\frac{m+1}{m-1}}$. Since $g(r) = J_p\left(
\pm \sqrt{\frac{2}{p+1}}r+C)\right)$ we have two possible types of solution depending on which sign we
take. If we choose the plus sign,
\[
	F(r) = \ell^{\oo{m-1}} \left[ J_p\left( \sqrt{\frac{2m}{m+1}} (|x| - R) \right)\right]^{\oo{m}},
\]
\[
	\rho(x,t) = \oo{[(m-1)(t_0-t)]^{\oo{m-1}}}\left[ J_p\left( \sqrt{\frac{2m}{m+1}}
		(|x|-R) \right)\right]^{\oo{m}}
\]
and
\[
	u(x,t) = \frac{m}{(m-1)^2(t_0 -t)} \left[ J_p\left( \sqrt{\frac{2m}{m+1}}
		(|x|-R) \right)\right]^{\frac{m-1}{m}}.
\]
For $R>0$, these functions are defined on $\{|x| \geq R\} \times (-\infty,t_0)$ with $u,\rho
\equiv +\infty$ the top portion of its boundary, $t=t_0$, $u,\rho = 0$ on the lateral boundary, $|x| =
R$, and again $Du,D\rho = +\infty$ on the lateral boundary. As $|x| \to \infty$, $\rho$ grows like
$|x|^{\frac{2}{m-1}}$ and $u$ grows like $|x|^2$. For $R=0$ the situation is identical. The function is
defined in all space but $D\rho$ has a singularity at $x=0$. Finally, if $R<0$, then $\rho$ is defined for
all $x$ but the gradient, though bounded near the origin, is discontinuous at $x=0$.

If we take a minus sign in the implicit identity for $g$, using the same inverse for $I_p$, we get
functions defined on a ball with a finite but discontinuous gradient at the origin.

\subsubsection{Case $a=0$}
The integral for $g$ is $\frac{2}{1-p} g^{\frac{1-p}{2}}(r) = \pm\sqrt{\frac{2}{p+1}} r + C$,
therefore
\[
	g(r) = \left[ \frac{1-p}{\sqrt{2(p+1)}} r + C \right]^{\frac{2}{1-p}}.
\]
In this case we have
\begin{equation}\label{E:sep_F3}
	F(x) = \ell^{\oo{m-1}} \left[ \frac{m-1}{\sqrt{2m(m+1)}} (|x|-k) \right]^{\frac{2}{m-1}},
\end{equation}
\[
	\rho(x,t) = \left( \frac{(m-1)^2}{2m(m+1)} \right)^{\oo{m-1}}
		\left[ \frac{(|x|-R)^2}{(t_0-t)} \right]^{\oo{m-1}}
\]
and
\[
	u(x,t) = \oo{2(m+1)} \frac{(|x|-R)^2}{(t_0-t)}.
\]
For $R>0$ these functions can be defined either on the interior or the exterior of the ball $\{|x| \leq R\}$.
For $R < 0$ they are smooth in the whole space minus the origin and for $R=0$ the function $u$ is a smooth solution
of \eqref{E:mu}.

\subsubsection{Case $a<0$}
In this case the integral of $g$ is $K_p(g) = \pm \sqrt{\frac{2}{p+1}} r + C$, where $K_p:[|a|^{\oo{p+1}},
+\infty) \to [0,+\infty)$ satisfies $K_p{'}(z) = \oo{\sqrt{z^{p+1} - |a|}}$ and $K_p(a^{\oo{p+1}}) =
0$. This function grows at infinity like $z^{\frac{1-p}{2}}$ and at $z=|a|^{\oo{p+1}}$ we have $K_p =
0$ and $K_p{'}(z) \sim \oo{\sqrt{(p+1)(z-|a|^{\oo{p+1}})}}$. If we denote by $L_p:[0,+\infty) \to
[|a|^{\oo{p+1}},+\infty)$ the inverse of $K_p$, we have $L_p(0) = |a|^{\oo{p+1}}$, $L_p{'}(z) \sim
\sqrt{(p+1)}z$ as $z\to 0$, while as $z\to +\infty$, $L_p(z) \sim z^{\frac{2}{1-p}}$ and $L_p{'}(z)
\sim z^{\frac{3+p}{1-p}}$. We then have
\begin{equation}\label{E:sep_F4}
	F(x) = l^{\oo{m-1}} \left[ L_p \left( \pm\sqrt{\frac{2m}{m+1}} |x| + C\right) \right]^{\oo{m}}
\end{equation}
and
\[
	\rho(x,t) = \oo{[(m-1)(t_0-t)]^{\oo{m-1}}} \left[ L_p \left( \pm\sqrt{\frac{2m}{m+1}} |x| + C \right)
		\right]^{\oo{m}}.
\]
Since $L_p(z)$ is like a power $z^2$ at $z=0$, the conclusions regarding $\rho$ and $u$ are similar to the
case $a=0$, we have functions defined on the interior and on the exterior of $\{|x|=r\}$.


\section*{Comments and open problems}

We do not analyze here the fine behavior nor regularity for the free boundary. This might in fact be a key component for a full uniqueness result as it happens with the (PME) (see \cite{CV}), which is possibly the main open problem related to this equation.

Another idea we do not explore is how to develop a numerical scheme, both for the equation and the free boundary.

We are planning to pursue the ideas of interpolated models mentioned in the introduction separately.

\medskip
\vskip .5cm

\noindent \textsc{Acknowledgments.} The authors are grateful for the support given by the
Spanish Project MTM2008-06326-C02-01. The first author is also supported by the Portuguese Project
PTDC-MAT-098060-2008.


\end{document}